\documentclass[11pt]{article}
\usepackage[utf8]{inputenc}
\usepackage{amsfonts, amsmath}
\usepackage{amssymb}
\usepackage{amsthm,amsmath,amscd}
\usepackage{enumerate}
\usepackage{url}
\usepackage[pdftex]{graphicx}
\usepackage[margin=25pt,font=small,labelfont=bf]{caption}
\usepackage{subfig}
\usepackage[numbers,square,sort&compress]{natbib}
\usepackage[colorlinks=true]{hyperref}
\newcommand{\defn}[1]{\textcolor{blue}{\emph{#1}}}
\newcommand*{\doi}[1]{doi: \href{https://dx.doi.org/#1}{\urlstyle{rm}\nolinkurl{#1}}}
\newcommand*{\arxiv}[1]{arXiv:  \href{https://arxiv.org/abs/#1}{\urlstyle{rm}\nolinkurl{#1}}}
\let\oldproofname=\proofname
\renewcommand{\proofname}{\rm\bf{\oldproofname}}

\newcommand{\RR}{\mathbb R}
\newcommand{\CC}{\mathbb C}

\newcommand{\QQ}{\mathbb Q}

\newcommand{\kk}{\mathbf{k}}
\newcommand{\R}{\mathbb R}

\newcommand{\ii}{\mathbf i}
\newcommand{\bna}{\begin{eqnarray}}
\newcommand{\ena}{\end{eqnarray}}
\newcommand{\ba}{\begin{eqnarray*}}
\newcommand{\ea}{\end{eqnarray*}}
\newcommand{\bs}[1]{}
\newcommand{\iprod}[2]{\left\langle {#1}, {#2}\right\rangle}

\DeclareMathOperator{\sing}{Sing}

\newtheorem{theorem}{Theorem}[section]
\newtheorem{corollary}[theorem]{Corollary}
\newtheorem{lemma}[theorem]{Lemma}
\newtheorem{proposition}[theorem]{Proposition}
\newtheorem{remark}[theorem]{Remark}
\newtheorem{definition}[theorem]{Definition}

\newtheorem{question}[theorem]{Question}

\newcommand{\AAA}{{\mathcal A}}

\newcommand{\IF}{{\rm IF}}
\newcommand{\IR}{{\rm IR}}
\newcommand{\ST}{{\rm Str}}
\newcommand{\GOR}{{\rm GOR}}
\newcommand{\LL}{{\rm LSS}}
\def\p{{\bf p}}
\def\q{{\bf q}}
\def\0{{\bf 0}}

\def\pn{\p =(\p_1, \dots, \p_{n}) }

\def\M{{\bf M}}
\def\H{{\bf H}}
\def\J{{\bf J}}
\def\I{{\bf I}}
\let\oldv=\v
\def\v{{\bf v}}

\def\u{{\bf u}}
\def\x{{\bf x}}

\def\y{{\bf y}}

\usepackage[margin=1in]{geometry}
\begin{document}
\title{Generically globally rigid graphs have generic universally
rigid frameworks\footnote{An earlier version of this paper
\cite{CGT16v1} had the title ``Generic global and universal
rigidity''.}}

\author{Robert Connelly
\thanks{Partially supported by NSF grant DMS-1564493}
\and 
Steven J. Gortler
\thanks{Partially supported by NSF grant DMS-1564473}
\and 
Louis Theran}
\date{}
\maketitle 

\begin{abstract}  
We show that any graph that is generically globally rigid  in $\R^d$ has a
realization in $\RR^d$ that is both generic and universally rigid. 
This also implies that the graph
also must have a realization in $\RR^d$ that is both infinitesimally 
rigid and universally rigid;
such a realization serves as a certificate 
of generic global rigidity.  

Our approach
involves an algorithm by Lov\'asz, Saks and Schrijver
that, for a sufficiently connected graph, 
constructs a general position orthogonal
representation of the vertices, and a result of Alfakih that
shows how this representation leads to a stress matrix and a
universally rigid framework of the graph.

\end{abstract}
\section{Introduction} \label{section:introduction}

In this paper we clarify one central aspect in the 
relationship between global and universal rigidity of frameworks of a graph.

Given a graph $G$ (with $n$ vertices and $m$ edges)
and a configuration $\pn$ of its vertices in 
$\RR^d$, we refer to the pair $(G,\p)$ as a \defn{framework}, and 
measure the Euclidean lengths along the edges of $G$ between pairs of 
vertices in $\RR^d$.
We call two frameworks, $(G,\p)$ and $(G,\q)$
\defn{congruent} if there is an isometry of all of
$\R^d$ that takes $\q$ to $\p$. This is equivalent to the property that
the Euclidean lengths are preserved between \emph{all pairs} of points
in a configuration.

We say that $(G,\p)$ is 
\defn{globally rigid} in $\RR^d$
if every framework, $(G,\q)$ in $\RR^d$, with the same edge lengths as $(G,\p)$, is congruent to $(G,\p)$.

We say that $(G,\p)$ is 
is \defn{universally rigid} if every framework, $(G,\q)$ with the same edge lengths as $(G,\p)$ in \emph{any}
dimension $\R^D$ is congruent to $(G,\p)$.

We say that a graph $G$ is \defn{generically globally rigid} (GGR) in
$\RR^d$ if every ``generic framework'' of $G$ in $\RR^d$ (you can
think of this as ``almost every'' framework in $\RR^d$) is globally
rigid.  It turns out that if a graph is not generically globally
rigid, then every ``generic framework'' of $G$ in $\RR^d$ is not
globally rigid \cite{Gortler-Thurston}.

The universal rigidity of frameworks in $\RR^d$ of $G$ 
does not  have such a simple behavior. 
There are  graphs with
Euclidean open (positive measure) sets of frameworks that are universally rigid, 
and other open sets of
frameworks that are not universally rigid
 (see, e.g., \cite[Remark 1.7]{Gortler-thurston2}).  
For example, for the line $\R^1$, and when the graph $G$ is a cycle, 
the only universally rigid configurations are when one edge length is the sum 
of the others, although all the generic configurations are globally 
rigid in the line (see \cite{JN15} for more about
universal rigidity in $\R^1$).

If one framework $(G,\p)$ is universally rigid, then clearly 
this framework is
globally rigid in $\RR^d$, but this does not imply that the graph $G$ itself,
is generically globally rigid in $\RR^d$. Indeed $\p$ might be somehow
exceptional, and not representative of the generic behavior of
frameworks of $G$ in $\RR^d$.  Figure \ref{fig:Exceptional} shows two examples of 
frameworks on $K_{3,3}$ 
that are universally rigid in the plane 
because the vertices are not separable by a quadric
\cite[Theorem 4.4]{CG17-bipartite} (see also
\cite{Connelly-energy}).  Generically, 
$K_{3,3}$ is minimally rigid, and thus not globally 
rigid \cite{Hendrickson}.  

 \begin{figure}[h]
     \begin{center}
          \includegraphics[width=0.4\textwidth]{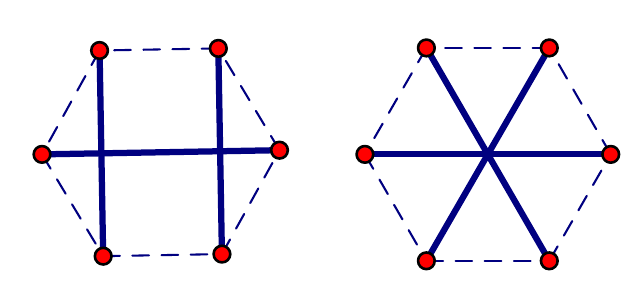}%
         \end{center}
     \caption{Universally rigid frameworks of a graph that is
not generically globally rigid.  The edges are drawn to indicate 
the signs of the entries of a PSD equilibrium stress matrix 
(see Section \ref{sec:background}) for each framework:
dashed lines correspond to negative entries and thick 
ones to positive entries.}
     \label{fig:Exceptional}
 \end{figure}
    
On the other hand, if we can find a Euclidean open set of
configurations of $G$ in $\RR^d$ that are universally rigid (and thus
globally rigid), then $G$ is generically
globally rigid in $\RR^d$~\cite{Gortler-Thurston}.  (We may replace
``open set of configurations'' by either ``a single generic
framework'' or ``a single infinitesimally rigid framework'' without
changing the conclusion.)

In this paper we show the converse. Namely, if  $G$ is
generically globally rigid in $\RR^d$, then it has a Euclidean open set
of frameworks in $\RR^d$ that are universally rigid.
This answers a question posed by Gortler and 
Thurston~\cite{Gortler-thurston2} for $d\ge 3$. Our method applies for $d\ge 1$, 
but the cases $d=1$ and $d=2$ have already been settled~\cite{Jordan-Szabadka,Jacobs}.  Both \cite{Jordan-Szabadka} and \cite{Jacobs}
rely, in a fundamental way, on the combinatorial
classification of GGR graphs for $d=2$
\cite{Berg-Jordan,Jackson-Jordan} which does not apply to higher
dimensions.

Our approach is to analyze a construction due to Alfakih~\cite{Alfakih-conn},
which builds on work of 
Lov\'asz, Saks and Schriver~\cite{Lovasz-Schrijver}.
The main result of \cite{Lovasz-Schrijver} is that any $(d+1)$-connected graph
admits an $(n-d-1)$-dimensional ``orthogonal representation'' in general position.
Alfakih \cite{Alfakih-conn} showed how to convert these representations into positive semidefinite (PSD)
stress matrices of rank $n-d-1$, which then
yield universally rigid frameworks
$(G,\p)$.

Since $(d+1)$-connectivity is strictly weaker 
than generic global rigidity \cite{Hendrickson}, there are graphs $G$ for which 
\emph{all} the universally rigid frameworks $(G,\p)$ constructed by Alfakih's
method are infinitesimally flexible and thus
lie in a proper algebraic subset of configurations.  
Our main result says that this does not happen when $G$ is GGR.

\section{Background}\label{sec:background}

Let $G$ be a graph with $n$ vertices and $m$ edges. Let $d$ be a fixed
dimension. Throughout, we will assume that $n\geq d+2$.

A \defn{(bar and joint) framework} in $\RR^d$, denoted as 
$(G,\p)$, is a graph $G$ together with a
configuration $\p=(\p_1, \dots,
\p_n)$ of points in $\R^d$.

\subsection{Rigidity of Frameworks}

We say that the framework $(G,\p)$
is \defn{locally rigid} in $\RR^d$ if,
except for congruences, 
there are no continuous motions 
in $\RR^d$ of the
configuration $\p(t)$, for $t \ge 0$, that preserve the
edge lengths:
\begin{equation}
|\p_i(t) - \p_j(t)|=|\p_i - \p_j|
\end{equation}
for all edges, $\{i, j\}$, of $G$, where $\p(0)=\p$.
If a framework is not locally rigid
in $\RR^d$, it is called \defn{locally flexible} in $\RR^d$ or equivalently just \defn{flexible} or a \defn{finite mechanism}.

The simplest way to confirm that a framework is locally rigid 
in $\RR^d$
is look at the 
linearization of the problem.

A \defn{first-order flex} or \defn{infinitesimal flex} of $(G,\p)$ 
in $\RR^d$
is a
corresponding assignment of vectors $\p'=(\p_1', \dots, \p_n')$,
$\p'_i \in \RR^d$ 
such
that for each $\{i,j\}$, an edge of $G$, the following
holds:
\begin{eqnarray}
(\p_i - \p_j)\cdot (\p'_i - \p'_j)&=&0 \label{eqn:first} 
\end{eqnarray}

A first-order flex in $\RR^d$
$\p'$ is \defn{trivial} if it is the restriction to the vertices, of the
time-zero derivative of a smooth
motion of isometries of $\R^d$. 
The property of being trivial is independent of the graph $G$.

The \defn{rigidity matrix} $R(\p)$ is the $nd$-by-$m$ matrix, where 
\[R(\p)\p'=(\dots, (\p_i - \p_j)\cdot (\p'_i - \p'_j), \dots)^T, \]
for $\p' \in \R^{nd}$,

A framework $(G,\p)$ in $\RR^d$ 
is called \defn{infinitesimally rigid} in $\R^d$
if it has no infinitesimal flexes in $\R^d$ except for trivial ones.
When $n \ge d$ this is the same as saying that the rank of $R(\p)$ is 
$nd-\binom{d+1}{2}$.
If a framework is not infinitesimally rigid in $\RR^d$, 
it is called \defn{infinitesimally flexible} in $\RR^d$.

A classical theorem states:
\begin{theorem} 
If a framework $(G,\p)$ is infinitesimally rigid in $\R^d$, then it is
locally rigid in $\R^d$.
\end{theorem}
The converse is not true (but see Theorem~\ref{thm:glr} below).

A framework $(G,\p)$ in $\RR^d$ is called \defn{globally rigid in $\R^d$} if,
there are no other
other (even distant)
frameworks $(G,\q)$ in $\R^d$ having the same  edge
lengths as $(G,\p)$, other than congruent frameworks.

A framework $(G,\p)$ in $\RR^d$ 
is called \defn{universally rigid} if,
there are no other
other (even distant)
frameworks $(G,\q)$ in $\R^D$, for any $D$,
having the same  edge
lengths as $(G,\p)$, other than congruent frameworks in $\RR^D$.

Clearly
universal rigidity implies global rigidity (in any dimension)
which implies
local rigidity (in any dimension).

Given a graph $G$,
 a \defn{stress vector} $\omega =( \dots,
 \omega_{ij}, \dots )$, is an assignment of a real scalar
 $\omega_{ij}=\omega_{ji}$ to each edge, $\{i,j\}$ in $G$.  (We have
 $\omega_{ij}=0$, when $\{i,j\}$ is not an edge of $G$.)

We say that 
 $\omega$ is an \defn{equilibrium stress vector}
for $(G,\p)$ if the vector equation
\begin{equation}\label{eqn:equilibrium}
 \sum_j \omega_{ij}(\p_i-\p_j) = 0
 \end{equation}
holds for all vertices $i$ of $G$.  
The equilibrium stress vectors of $(G,\p)$ form the co-kernel of its
rigidity matrix $R(\p)$.

We associate an $n$-by-$n$
\defn{stress matrix} $\Omega$ to a stress vector $\omega$, 
by setting the $i,j$th entry of $\Omega$ to
$-\omega_{ij}$, for $i \ne j$, and the diagonal entries of $\Omega$
are set such that the row and column sums of $\Omega$ are zero.  
The stress matrices of $G$  
are simply the symmetric matrices with zeros associated to
non-edge pairs that, additionally, have the all-ones vector 
in their kernel.
 
If $\omega$ is an equilibrium stress vector for
$(G,\p)$ then we say that the associated $\Omega$ is an 
\defn{ equilibrium stress matrix}
 for $(G,\p)$. For each of the  $d$ spatial dimensions,
if we define a vector $\v$ in $\RR^n$ by collecting  the  associated 
coordinate over all of the points in $\p$, we have $\Omega \v=0$.
Thus if the 
dimension of the affine span of the vertices $\p$ is $d$, then the
rank of $\Omega$ is at most $n-d-1$, but it could be less.

Let $(G,\p)$ be a framework (in any dimension) with a $d$-dimensional affine
span, denoted $\langle \p \rangle$.
Fixing an affine frame for $\langle \p \rangle$,
we can represent $\p$ using 
coordinates in $\RR^d$.
We say that the edges directions 
of $(G,\p)$ lie on a conic at infinity of $\langle \p \rangle$  
if there exists
a non-zero symmetric $d$-by-$d$ matrix $Q$ such that for all of the
edges,
$\{ij\}$ in $G$, 
we have $(\p_i-\p_j)^t Q (\p_i-\p_j)=0$.

Following \cite{Connelly-energy} we say a framework $(G,\p)$ 
(in any dimension)
with a $d$-dimensional affine span
is
\defn{super stable} if there is an equilibrium stress $\omega$ for
$(G,\p)$ such that its associated stress matrix $\Omega$ is PSD, the
rank of $\Omega$ is $n-d-1$,
and the edge directions do not
lie on a conic at infinity of $\langle \p \rangle$.

The following is a classic theorem by Connelly
\cite{Connelly-energy}
\begin{theorem}
Let $(G,\p)$ be a framework (in any dimension).
If $(G,\p)$ is super stable
then $(G,\p)$ is 
universally rigid.
\end{theorem}

Alfakih and Ye~\cite{Alfakih-Ye-general-position}, 
showed that one can easily 
avoid the explicit assumption about conics at infinity in the case
of general position.  
\begin{theorem}
Let $(G,\p)$ be a framework
with a $d$-dimensional affine span.
If $(G,\p)$ is in general affine position within $\langle \p \rangle$
and
has an (even indefinite) equilibrium stress matrix of rank $n-d-1$,
then the edge directions of $(G,\p)$ do not lie on a conic at 
infinity of $\langle \p \rangle$.  

Thus if $(G,\p)$ is a framework 
with a $d$-dimensional affine span and in 
general affine position within $\langle \p \rangle$
and it
has a PSD equilibrium stress matrix of rank $n-d-1$, then
it is super stable and thus universally rigid. 
\end{theorem}

\subsection{Rigidity of Graphs}

We say that
a configuration $\p$, or a framework $(G,\p)$, in $\RR^d$
is \defn{generic}, if there
is no non-zero  polynomial relation, with coefficients in $\QQ$,
among the coordinates of
$\p$.

We say that a graph $G$ is 
\defn{generically locally rigid (resp. flexible)} in $\RR^d$ if 
every generic framework of $G$ in $\RR^d$ is locally rigid (resp. flexible)
in $\RR^d$. 

We say that a graph $G$ is 
\defn{generically infinitesimally rigid (resp. flexible)} in $\RR^d$ if 
every generic framework of $G$ in $\RR^d$ is 
infinitesimally rigid (resp. flexible) in $\RR^d$.

As described in~\cite{Asimow-Roth-I,Asimow-Roth-II},
generic local rigidity is determined by generic infinitesimal 
rigidity
\begin{theorem}
\label{thm:glr}
If some framework $(G,\p)$ in $\RR^d$
is infinitesimally rigid in $\RR^d$, then the graph 
$G$ is generically infinitesimally rigid in $\RR^d$ and thus
generically 
locally rigid in $\RR^d$.

If a  graph, $G$, is not
generically infinitesimally rigid in $\RR^d$ then it is 
generically locally flexible in $\RR^d$.

Thus, if $G$ is not generically locally rigid in $\RR^d$
then it is 
generically locally  flexible in $\RR^d$.

\end{theorem}

We say that a graph $G$ is 
\defn{generically (resp. not) globally rigid in $\RR^d$} if 
every generic framework of $G$ in $\RR^d$ is (resp. not) globally rigid
in $\RR^d$.

The following is the easy half of a theorem by 
Hendrickson~\cite{Hendrickson-unique}, which we will need below.
\begin{theorem}
\label{thm:hen}
If $G$ is generically globally rigid in $\RR^d$, then it must
be $(d+1)$-connected.
\end{theorem}

Connelly~\cite{Connelly-global} 
proved the following sufficient condition  for 
global rigidity.
\begin{theorem}
\label{thm:suff}
If some generic framework $(G,\p)$ in $\RR^d$
has an (even indefinite) equilibrium stress matrix 
of rank $n-d-1$, then the graph $G$ is
generically globally rigid in $\RR^d$.
\end{theorem}

This was refined slightly in~\cite{Gortler-Thurston,coning} 
giving the following sufficient certificate for 
generic global rigidity
\begin{theorem}
\label{thm:suff2}
If some framework $(G,\p)$ in $\RR^d$
is infinitesimally  rigid in $\RR^d$
and $(G,\p)$ has an (even indefinite) equilibrium stress matrix $\Omega$
of rank $n-d-1$, then the graph $G$ is
generically globally rigid in $\RR^d$.
\end{theorem}
Thus, the pair $\Omega$ and $(G,\p)$ serve as a certificate
for the generic global rigidity of $G$ in $\RR^d$.
Note that this does not imply that the specific framework $(G,\p)$ in
the above certificate is globally rigid in $\RR^d$~\cite{coning}.

Gortler Healy and Thurston~\cite{Gortler-Thurston} 
proved the strong converse to 
Theorem~\ref{thm:suff}.
\begin{theorem}
\label{thm:necc}
If some generic framework $(G,\p)$ in $\RR^d$
does not have equilibrium stress matrix 
of rank $n-d-1$, then the graph $G$ is not
generically globally rigid in $\RR^d$.
It is, in fact, generically not globally rigid in $\RR^d$.
Thus, if a graph $G$ is not generically globally rigid in $\RR^d$
then it is 
generically not globally rigid in $\RR^d$.
\end{theorem}

\begin{remark}
The above theorems tell us that that a graph $G$ is either 
generically (locally / infinitesimally / globally)
rigid 
in $\RR^d$, 
or it is 
generically not (locally/infinitesimally/globally)
rigid
in $\RR^d$.

Due to the semi-algebraic nature of rigidity, 
if $G$ is generically (resp. not) (locally / infinitesimally / globally)
rigid  in $\RR^d$, 
then the only exceptional frameworks 
must be contained in an strict algebraic  subset (defined over $\QQ$)
of configuration space. 

Universal rigidity does not behave so simply. 
In particular, there are  graphs with
Euclidean
open sets of frameworks in $\RR^d$
that are universally rigid, and other open sets of
frameworks in $\RR^d$
that are not universally rigid.  
\end{remark}

The examples above indicate that a  
graph can be generically globally rigid
in $\RR^d$, while having some generic frameworks in $\RR^d$ that 
are not universally rigid. One open question
that has been open in the rigidity community 
since 2010 (see~\cite{Gortler-thurston2}) asks:
\begin{center}
\emph{If $G$ is generically globally rigid in $\RR^d$, must it have	
\emph{some} generic framework in $\RR^d$ that is universally rigid?}
\end{center}

The main result of this paper answers this question 
in the affirmative:
\begin{theorem}
\label{thm:main}
If $G$ is generically globally rigid in $\RR^d$, then 
there exists a framework $(G,\p)$ in $\RR^d$ that is 
infinitesimally rigid in $\RR^d$ and super stable.
Moreover, every framework in a small enough neighborhood of $(G,\p)$
will be infinitesimally rigid in $\RR^d$ and 
super stable, and thus must include some generic framework.
\end{theorem}

The first part of this theorem tells us that if $G$ is generically
globally rigid in $\RR^d$, then it must have a certificate,
$\Omega$ and $(G,\p)$ 
in the sense of Theorem~\ref{thm:suff2}, where $(G,\p)$ is itself
certifiably super stable and thus globally rigid.

\begin{remark}
Theorem \ref{thm:main} yields a 
weak converse to 
Connelly's Theorem~\ref{thm:suff}. Namely, 
If some generic framework $(G,\p)$ in $\RR^d$
does not have an equilibrium stress matrix 
of rank $n-d-1$, then the graph $G$ is not
generically globally rigid in $\RR^d$.
But this does not, alone, prove that $G$
is, in fact, generically not globally rigid in $\RR^d$.
(This requires showing the existence of an equivalent, but not
congurent gramework for each generic $(G,\p)$.)
\end{remark}

This paper will also use some basic facts from (semi-)algebraic
geometry, which are summarized in the appendix.

\section{Stresses from GORs}

In~\cite{Lovasz-Schrijver}, Lov\'asz 
Saks and Schriver define a concept called a 
(GOR) general position
orthogonal representation of a graph $G$ in $\RR^{n-d-1}$.
Alfakih~\cite{Alfakih-conn} has shown how these
relate to a certain class of equilibrium stresses for $d$-dimensional
frameworks of $G$.  This section reviews and extends these results.

\subsection{GORs and connectivity}
\begin{definition}
\label{def:GOR}
Let $G$ be a graph and 
let $D$ be a fixed dimension.
An (OR)  \defn{orthogonal representation} of $G$ in $\RR^D$
is a vector configuration $\v$ indexed by the vertices of $G$ in 
$\RR^{D}$ with the following  property:
$\v_i$ is orthogonal to the vectors
associated with each non-neighbor of vertex $i$. 
The set of ORs form an algebraic set (defined over $\QQ$).

A (GOR) \defn{general position orthogonal representation} of $G$ 
in $\RR^D$
is an OR in $\RR^D$ with the added property that 
the $\v_i$ are in general linear position.
The set of GORs form a semi-algebraic 
set (defined over $\QQ$). 

\end{definition}
The relevant results from \cite{Lovasz-Schrijver} are the following.  
\begin{theorem}\label{thm:lss}
Let $G$, a graph on $n$ vertices,
be  $(n-D)$-connected for some $D$.  Then $G$ must
have a GOR in dimension $\R^D$ \cite[Theorem 1.1]{Lovasz-Schrijver}.  
Moreover, the set of all such GORs of $G$ is irreducible~\cite[Theorem 2.1]{Lovasz-Schrijver}.
\end{theorem}

In our terminology, we will set $D:=n-d-1$ where $d$ is fixed,
and thus we will need
$(d+1)$-connectivity to obtain GORs in $\RR^{n-d-1}$.

\begin{definition}
Let $G$ be a $(d+1)$-connected graph with $n$ vertices, 
for some $d$.  
Denote by $D_G$ the dimension of the set its GORs in $\RR^{n-d-1}$.
\end{definition}

We wish to compute $D_G$ which is done in the following corollary proven 
below.
\begin{corollary}\label{cor:gor-dim}
Let $G$ be a $(d+1)$-connected graph with $n$ vertices and 
$m$ edges.  Then the dimension $D_G$  is $n(n-d)-\binom{n+1}{2}+m$.
\end{corollary}

The idea behind the corollary, which is  present in \cite{Lovasz-Schrijver},
is that we can build a GOR of $G$ by selecting vectors one at a time from a 
linear space of known dimension (that depends on $G$ and the vertex order).  The proof of
the corollary relies on several lemmas that formalize this intuition.
\begin{definition}
 Let $G$ be a graph with $n$ ordered vertices $\{1,2,\ldots,n\}$.  
 Fix $d$. 
 Let $G_i$, for $2 \le i\le n$ be the subgraph of $G$ induced by vertices $j$ such that 
$j \le i$.

Let $\GOR_{i-1}$ be the set of GORs of  $G_{i-1}$ in $\RR^{n-d-1}$,
where $n$ is the number of vertices on the full graph $G$.

Fix $\v^{i-1}$, some  configuration in $\GOR_{i-1}$.
Let $A\subset \R^{n-d-1}$ be the linear span of the $\v_j$ 
in $\v^{i-1}$
corresponding 
to non-neighbors of vertex $i$
in $G_{i-1}$.  
We say that $\v^{i-1}$ is \defn{inextendable} if there is a set of
$\v_j$ in $\v^{i-1}$ of cardinality at most $n-d-2$ such that
$A^\perp$ is in the span of these $\v_j$. Otherwise we say that 
$\v^{i-1}$ is \defn{extendable}.
Every extendable configuration
in $\GOR_{i-1}$ can be extended
to a configuration in $\GOR_i$ by some appropriate placement of
vertex $i$ in $A^\perp$.

Let
$\bar{e}_i$ denote the number of vertices
in $G_{i}$ that are not neighbors of vertex
$i$ and $e_i$ be the number of its neighbors
in $G_i$.
\end{definition}

\begin{lemma}
\label{lem:gori}
Let $G$ be a $(d+1)$-connected graph with $n$ vertices, 
for some $d$.  
Then for any $1 \le i \le n$, we have $\GOR_i$ is non-empty and
irreducible.
\end{lemma}
\begin{proof}
The graph $G_i$ has $i$ vertices
and must be at least
$\left(d+1-(n-i)\right)$-connected. 
(Negative connectivity is the same
as having no connectivity conditions at all).
Meanwhile, in order to apply Theorem~\ref{thm:lss} directly to $G_i$,
we only need it to be $\left(i-(n-d-1)\right)$-connected.
\end{proof}

\begin{lemma}
\label{lem:dimA}
The subspace $A^\perp$ has dimension 
\[
	(n-d-1)-\bar{e}_i = (n-d-1)-(i-1-e_i)
=n-d-i+e_i
\]
\end{lemma}
\begin{proof}
By assumption, the 
$\bar{e}_i$ non-neighbors are in 
general position, giving us the first
expression. The rest is obvious.
\end{proof}

\begin{lemma}
\label{lem:dom}
Let $2 \le i \le n$.
The subset of $\GOR_{i-1}$ that is inextendable is semi-algebraic
and of strictly lower  dimension than $\GOR_{i-1}$. Thus the
extendable subset has full dimension.
\end{lemma}
\begin{proof}
The conditions describing inextendibility can be described with
algebraic equations (using determinants).
Also, from Lemma~\ref{lem:gori}, $\GOR_{i-1}$ is irreducible. Thus
the inextendable set is either all of $\GOR_{i-1}$ or it is of lower
dimension.

Meanwhile, from Lemma~\ref{lem:gori}
$\GOR_i$ is not empty, thus it contains
some configuration
$\v^i$. By forgetting the last 
vertex, we obtain a configuration $\v^{i-1}$
in $\GOR_{i-1}$ that must be extendable.
\end{proof}

\begin{lemma}
\label{lem:fullA}
Let $2 \le i \le n$.
Suppose that $\v^{i-1}$ is an extendable
configuration in 
$\GOR_{i-1}$. Then there is a Zariski
open subset of $A^\perp$ such that placing
$\v_i$ in this subset produces
an element in $\GOR_i$.
\end{lemma}
\begin{proof}
The set of disallowed placements for $\v_i$ 
(violating general position)
is the intersection 
of the irreducible $A^\perp$ with a 
subspace arrangement arising from the 
linear spans of subsets of $\v^{i-1}$.  Either $A^\perp$
is contained in this arrangement, or the disallowed subset is 
algebraic and lower dimension.  
\end{proof}

\begin{proof}[Proof of Corollary \ref{cor:gor-dim}]
Let $2 \le i \le n$.
Let $\pi$ be the 
map $\pi:\GOR_{i}\to \GOR_{i-1}$ that forgets the last vertex.
From Lemma~\ref{lem:dom}, 
we have $\dim(\pi(\GOR_i))=\dim(\GOR_{i-1})$.

Due to irreducibility of $\GOR_i$ (Lemma~\ref{lem:gori}), we can apply 
the fiber dimension theorem
~\ref{thm:fiber} in the appendix to see that
\[
\dim(\GOR_i)= \dim(\pi(\GOR_i)) + \dim(\pi^{-1}(\pi(x)) \cap N(x))
\]
where $x$ is generic in $\GOR_i$ and $N(x)$ is a neighborhood around $x$.

Meanwhile, from Lemmas~\ref{lem:fullA}
and~\ref{lem:dimA}
any fiber
$\pi^{-1}(\pi(x))$ is a Zariski open subset of a linear space of dimension 
$n-d-i+e_i$. 
Thus the dimension of $\GOR_i$ is $n-d-i+e_i$ 
more than the dimension of 
$\GOR_{i-1}$.
Also the dimension of $\GOR_i$ for $i=1$ is 
$n-d-1=n-d-i+e_i$.

Summing over all $i$ gives 
\[
D_G = \sum_{i=1}^{n} n-d -i + e_i = n(n-d) - \binom{n+1}{2} + m
\]
as claimed.
\end{proof}

\subsection{Alfakih's construction}
Because of the orthogonality property of a GOR, its 
Gram matrix has the right zero/non-zero pattern 
to be a stress matrix.  Alfakih \cite{Alfakih-conn}
builds on this.  First we set some notation.
\begin{definition}\label{defn:gorc}
Let $G$ be a $(d+1)$-connected 
graph and $\v$ a GOR of $G$ in dimension 
$n - d -1$.  The $n\times (n - d - 1)$ matrix $X$
with the $\v_i$ as its rows is the \defn{configuration 
matrix} of $\v$.  We denote the Gram matrix $XX^t$ of 
$\v$ by $\Psi$.  Note that $\Psi$ is, by construction, PSD
and  has 
rank $n - d - 1$, (as $\v$ is in general position). 

A GOR $\v$ is called \defn{centered} if its 
barycenter is the origin.  We define 
$\GOR^0$ to be the semi-algebraic set of 
centered GORs.

The Gram matrix $\Psi$
is a stress matrix (which we will call $\Omega$) 
if and only if $\v$ is centered.  (Recall that the 
extra condition is that the all-ones vector is in the 
kernel.) Such an $\Omega$ is  PSD and of 
rank $n-d-1$.

We define the set $\LL$ of \defn{Lovász-Saks-Schrijver stresses}
to be the collection of stress matrices $\Omega$ arising 
as the Gram matrices of centered GORs.  Denote its
dimension by $D_L$.
\end{definition}
We wish to compute $D_L$.  
Heuristically, we expect the relationship
\begin{equation}\label{eq:gorc-dim}
	D_L + \binom{n - d - 1}{2} = D_G - n + d + 1
\end{equation}
to hold
because both sides correspond to the dimension 
of the set $\GOR^0$ of centered GORs.
In particular, given 
a stress matrix $\Omega\in \LL$, we can change
the underlying GOR $\v$ by an orthogonal transformation 
on $\RR^{n - d -1}$ without changing $\Omega$.  This 
corresponds to the left-hand side of \eqref{eq:gorc-dim}.
The
right-hand-side comes from noting that the centering 
condition imposes a linear constraint on each column
of the matrix $X$.  

This this does not constitute a proof because we don't yet know 
that the centering condition behaves transversely, 
which we need to prove
correctness of the predicted count.
Instead of checking this directly,
we use the construction from \cite{Alfakih-conn}.

\begin{definition}\label{def:gorc-scaling}
Let $\v$ be a vector configuration.  
A \defn{centering map} $\varphi$ is a map $\v_i\mapsto \alpha_i\v_i$
so that $\varphi(\v)$ has its barycenter at the origin.  A 
centering map is \defn{full rank} if none of the $\alpha_i$ are
zero.  The coefficients $\alpha_i$ defining $\varphi$ correspond
to a row vector in the co-kernel of the configuration matrix, $X$, of $\v$.

We define $\varphi(X)$ to be 
$DX$, where $D$ is the $n\times n$ diagonal matrix 
with the $\alpha_i$ on its diagonal.

Specialized to the case where $\v$ is a GOR for a graph $G$ in 
$\RR^{n-d-1}$, $\varphi(\v)$ is a centered GOR 
if and only if $\varphi$ is full rank, since general position 
must be maintained.  

Thus the set of full rank centering maps is the semi-algebraic set 
arising by removing vectors with any zero coordinates from the 
co-kernel of $X$.
\end{definition}

\begin{lemma}\label{lem:centering-map-dim}
Let $G$ be a $(d+1)$-connected graph and $\v$ a GOR of $G$ in 
 $\RR^{n-d-1}$.  If there is at least one full-rank centering 
map $\varphi$ of $\v$, then the set of all full rank centering 
maps is $(d+1)$-dimensional.
\end{lemma}
\begin{proof}
The co-kernel $K$ of $X$ is a linear space, and so irreducible.  The
subset of vectors in $K$ with any zero coordinate is an algebraic 
subset of it, and so is either all of $K$ or of lower dimension.
Since $K$ is assumed to contain a vector with no zeros, the set of 
full rank centering maps is then the (semi-algebraic) complement 
of a proper algebraic subset, so it has the same dimension as $K$.

Because $\v$ is in general position, the dimension of $K$ is $d+1$.
\end{proof}
Alfakih's main result in \cite{Alfakih-conn} is the following.
\begin{theorem}[\cite{Alfakih-conn}]
\label{thm:alf}
Let $G$ be $(d+1)$-connected.  Then any GOR
in $\RR^{n-d-1}$ for $G$ has a full rank centering map.
This gives rise to stress matrix $\Omega$ in LSS.  Moreover, 
any framework $(G,\p)$ with $d$-dimensional affine span, that has 
$\Omega$ as an equilibrium stress matrix, must 
be in general position.  (And, consequently,
super stable.)
\end{theorem}
\begin{remark}
In Theorem \ref{thm:alf}, 
the existence of the full rank centering map relies crucially on the general position property of the
input GOR.
The existence of a general position kernel framework 
relies  
on the general position of the centered GOR 
obtained by scaling.  A centered OR that is not
in general position but still has rank $n-d-1$
has only non-general position frameworks in its kernel.
\end{remark}

\begin{corollary}
\label{cor:dl}
If $G$ is  $(d+1)$-connected, then $D_L=m-\binom{d+1}{2}$.
\end{corollary}
\begin{proof}
If we can establish \eqref{eq:gorc-dim}, then we have 
\ba
    D_L &=& D_G - \binom{n-d-1}{2} - n + d + 1 
    \\ &=& 
    m + n(n-d) - \binom{n+1}{2} - \binom{n-d-1}{2}
    - n + d + 1
    \\&=&
    m-\binom{d+1}{2}
\ea

Now we show \eqref{eq:gorc-dim} by computing the 
dimension of $\GOR^0$ two ways.

For the first way, 
we build a semi-algebraic bundle $B$ of points $(\v, \varphi)$,
where $\v$ is a GOR and $\varphi$ is full rank centering map for $\v$.
Combining Lemma \ref{lem:centering-map-dim} and Theorem \ref{thm:alf},
for each fixed $\v$ the set of $\varphi$ is $(d+1)$-dimensional.
Thinking of $B$ as a bundle $(X^t,\varphi)$, we may apply 
Lemma~\ref{lem:bundle0}, to get that 
$B$ is irreducible and of dimension $D_G + d + 1$.

Any two GORs $\v$ and $\v'$,
can be scaled to the 
same element of $\GOR^0$ iff 
all of their corresponding vectors $\v_i$ and $\v'_i$
share the same direction.
Thus, the natural 
map $B\to \GOR^0$ given by $(\v, \varphi)\mapsto \varphi(\v)$
has $n$-dimensional fibers.  This maps is also surjective since
$\GOR^0\subset \GOR$, so $\v\in \GOR^0$ is in the image by taking 
$\varphi$ to be the identity.  We may now apply 
Theorem \ref{thm:fiber} again to conclude that, $\GOR^0$ has dimension 
$D_G - n + d + 1$. We also see, that as the image of this polynomial map, 
$\GOR^0$ is irreducible.

For the second way, 
as discussed above, the 
map from $\GOR^0\to \LL$, given by
$X\mapsto XX^t$ is invariant under the orthogonal 
group, so its fibers are $\binom{n-d-1}{2}$-dimensional.
This map is, by definition, surjective, so 
by Theorem \ref{thm:fiber}
the dimension of $\GOR^0$
is $D_L + \binom{n-d-1}{2}$.
\end{proof}

\begin{remark}\label{rem:scalings}
In Alfakih's construction, if one starts with a fixed GOR
and  varies the full rank centering maps $\varphi$, 
the resulting $\Omega$ matrices will differ only through scaling.
Thus all of the $d$-dimensional frameworks in the kernels of these 
$\Omega$ must only differ through $d$-dimensional projective 
transforms~\cite{coning}.
\end{remark}
\begin{figure}[h]
     \begin{center}
          \includegraphics[width=0.15\textwidth]{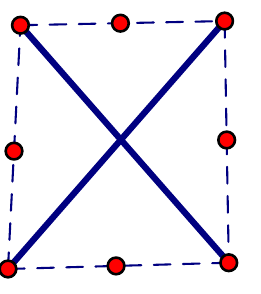}%
         \end{center}
     \caption{A framework with a maximal rank PSD equilibrium 
stress matrix that does not arise from Alfakih's construction.
The edge styles have the same meaning as in Figure \ref{fig:Exceptional}.}
     \label{fig:subdivided}
\end{figure}
\begin{remark}\label{rem:not-ss}
Since all the configurations produced by Alfakih's construction 
are in affine general position, and because it only applies to 
$(d+1)$-connected graphs, there are many frameworks 
with a maximal rank PSD equilibrium stress matrix that it does 
not construct.  For an example, see Figure \ref{fig:subdivided}.
\end{remark}
\begin{remark}
Alfakih's construction is an example of what the 
engineering literature calls ``form finding''
\cite{VB-form-finding}.  It is particularly similar 
to \cite{Connelly-Back}, in that it produces a configuration 
with specific properties by searching among configurations 
with a specific equilibrium stress.
\end{remark}

\section{Example: $K_{2,2}$ in $\RR^1$}
As a concrete example of Alfakih's construction, we consider 
the case of $K_{2,2}$ in $\RR^1$.  For convenience label the 
vertices $u_1,u_2,v_3,v_4$, so that the edges are $\{u_i,v_j\}$
for $i,j\in \{1,2\}$.  The universally rigid configurations
are those with one ``long edge'' and three ``short'' ones \cite{Jordan-Szabadka,Jacobs}.
We can explore these from the perspective of GORs for $K_{2,2}$ in 
dimension $\RR^2$ (since $n=4$ and $d=1$).  Denote a GOR of $K_{2,2}$ by $(\u_1,\u_2,\v_1,\v_2)$. 

The space of GORs is easy to describe: we have $\u_1\perp \u_2$, 
$\v_1\perp \v_2$, and the angle, $\theta$, between $\u_1$ and $\v_1$ is not 
a multiple of $\pi/2$.  This is clearly a $6$-dimensional set, 
in accordance with Theorem \ref{thm:lss}.

To explore the example more directly, we define a curve of
reference GORs parameterized by the angle $\theta$.
\[
\begin{array}{cccc}
    \u^\theta_1  := \ii e^{-\ii\theta/2}  &
    \u^\theta_2  := e^{-\ii\theta/2}
    & \v^\theta_1  := \ii e^{\ii\theta/2}
     & \v^\theta_2  := -e^{\ii\theta/2}
\end{array}
\]
where we have identified $\RR^2$ with $\CC$ to keep the formulas compact, 
and $\ii$ is the imaginary unit.

Following Alfakih's construction, we first scale our
reference curve of GORs to centered ones.  Noting the 
the reflection symmetry of this parameterization, we see that, for 
$\theta\in [0,\pi/2]$
\begin{eqnarray}
    \label{eq:scale}
    |\u^\theta_2 + \v^\theta_2 |\; (\u^\theta_1 + \v^\theta_1) +
    |\u^\theta_1 + \v^\theta_1 |\; (\u^\theta_2 + \v^\theta_2) & = & 0 
\end{eqnarray}
and
\begin{eqnarray}
    \label{eq:scale2}
    |\u^\theta_2 + \v^\theta_1|\;(\v^\theta_2 - \u^\theta_1) + 
    |\v^\theta_2 - \u^\theta_1|\; (\u^\theta_2 + \v^\theta_1) & = & 0
\end{eqnarray}
This  gives us  basis for scalings of a reference GOR to $\GOR^0$. 
These two scalings, along with $\theta$, parameterize 
$\LL$ for this interval, since there are three independent parameters.  
Thus, $K_{2,2}$ has  a $3$-dimensional space for $\LL$,
as expected from Corollary \ref{cor:dl}.

 \begin{figure}[ht]
     \begin{center}
           \includegraphics[width=0.6\textwidth]{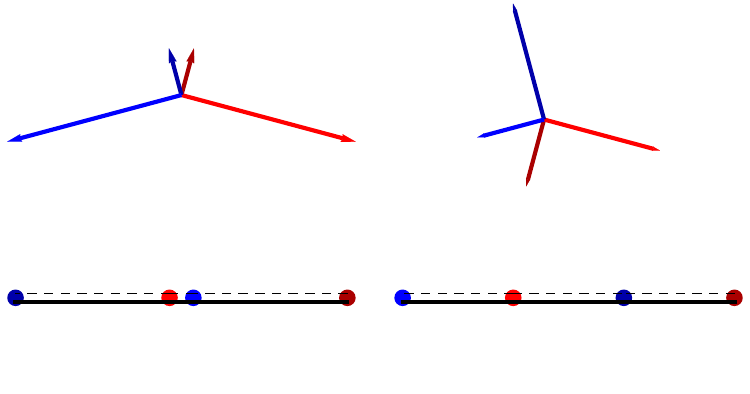}
     \end{center}
     \caption{Two centered GORs with $\theta=\pi/6$. Their kernel frameworks
are related by a  projective transformation.}
     \label{fig:scalings}
 \end{figure}
For a fixed $\theta$, applying a different scaling leads to a different 
point of $\GOR^0$.  The effect on the kernel framework, by 
Remark \ref{rem:scalings}, is to apply a projective 
transformation.  This is illustrated in Figure \ref{fig:scalings},
where the left column uses the scaling \eqref{eq:scale} and the 
right column uses \eqref{eq:scale2}.  
The color coding is: $u_1$ in dark red; $u_2$ in bright red; 
$v_1$ in dark blue; $v_2$ in bright blue.  
The top row shows two
centered GORs at $\theta=\pi/6$ and the bottom row the associated kernel 
frameworks. As is expected, in each case there is one red-blue
vector pair with positive dot product, corresponding to a negative
$\omega$ value on the long edge of that framework (represented
 with a thick line).
Edges with positive $\omega$ are shown with a dotted line and edges with 
zero $\omega$ are shown in thin green.

 \begin{figure}[ht]
     \begin{center}
         \subfloat[$\theta=0$]{
             \includegraphics[width=0.2\textwidth]{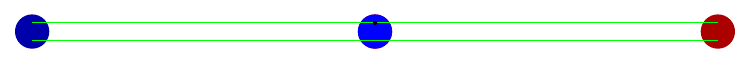}
         }
        \subfloat[$\pi/20$]{
            \includegraphics[width=0.2\textwidth]{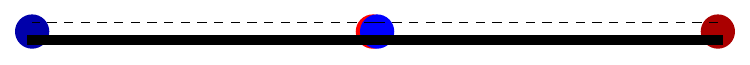}
         }
        \subfloat[$2\pi/20$]{
             \includegraphics[width=0.2\textwidth]{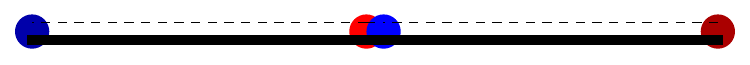}
         }
         \subfloat[$3\pi/20$]{
             \includegraphics[width=0.2\textwidth]{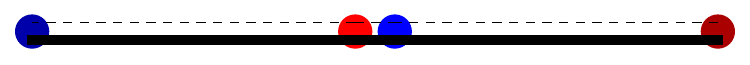}
         }
         \\
          \subfloat[$4\pi/20$]{
              \includegraphics[width=0.2\textwidth]{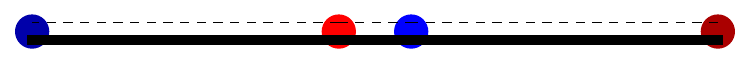}
          }
         \subfloat[$5\pi/20$]{
             \includegraphics[width=0.2\textwidth]{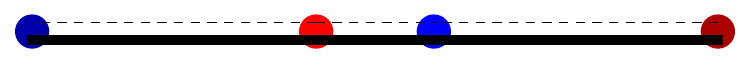}
          }
         \subfloat[$6\pi/20$]{
              \includegraphics[width=0.2\textwidth]{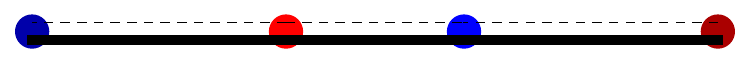}
          }
          \subfloat[$7\pi/20$]{
              \includegraphics[width=0.2\textwidth]{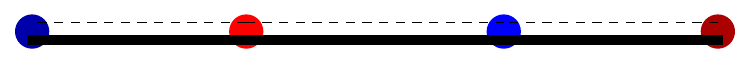}
          }
          \\
           \subfloat[$8\pi/20$]{
               \includegraphics[width=0.2\textwidth]{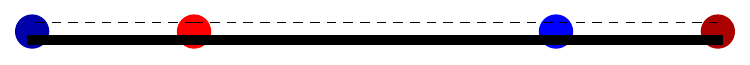}
           }
          \subfloat[$9\pi/20$]{
              \includegraphics[width=0.2\textwidth]{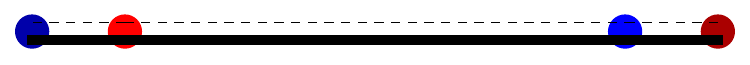}
           }
          \subfloat[$10\pi/20$]{
               \includegraphics[width=0.2\textwidth]{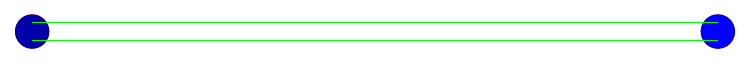}
           }
           \subfloat[$11\pi/20$]{
               \includegraphics[width=0.2\textwidth]{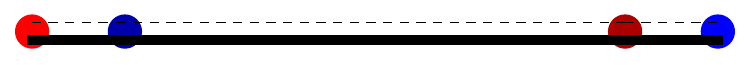}
           }
           \\
            \subfloat[$12\pi/20$]{
                \includegraphics[width=0.2\textwidth]{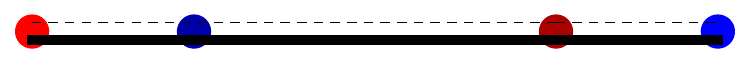}
            }
           \subfloat[$13\pi/20$]{
               \includegraphics[width=0.2\textwidth]{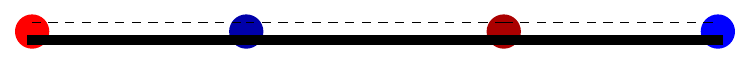}
            }
           \subfloat[$14\pi/20$]{
                \includegraphics[width=0.2\textwidth]{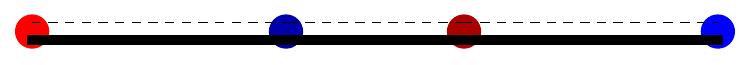}
            }
            \subfloat[$15\pi/20$]{
                \includegraphics[width=0.2\textwidth]{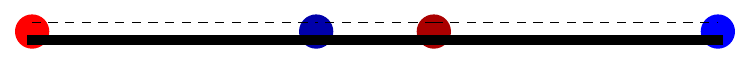}
            }
            \\
             \subfloat[$16\pi/20$]{
                 \includegraphics[width=0.2\textwidth]{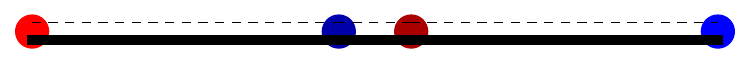}
             }
            \subfloat[$17\pi/20$]{
                \includegraphics[width=0.2\textwidth]{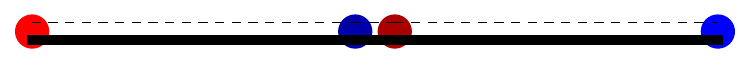}
             }
            \subfloat[$18\pi/20$]{
                 \includegraphics[width=0.2\textwidth]{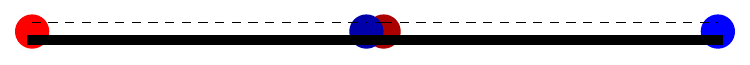}
             }
             \subfloat[$19\pi/20$]{
                 \includegraphics[width=0.2\textwidth]{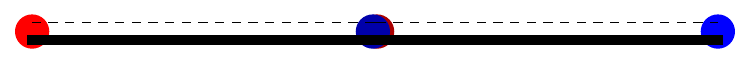}
             }
         \end{center}
     \caption{Universally rigid configurations of $K_{2,2}$ in $\RR^1$.}
     \label{fig:k22}
 \end{figure}

Figure \ref{fig:k22}
illustrates the sequence of kernel frameworks associated with the
scaling \eqref{eq:scale}, which holds on $[0,\pi]$,  
as $\theta$ increases from $0$ to $\pi$, using a 
consistent affine normalization.
At the endpoints of the 
interval, 
the ``GOR'' is no longer
in general position. Moreover, 
the ``centered GOR'' generated by this scaling
degenerates, yielding  a rank $1$ stress matrix;
in the figure, we show the kernel framework that is the limit 
as $\theta\to 0$ in part (a). 
The limit stress at $\theta=0$ has $\omega =1$ for the edge between 
bright red and bright blue and 
$\omega=0$ for the other edges.
In part (k), where $\theta = \pi/2$, the ``GOR'' is also no longer
in general position, but under the scaling, we still obtain a stress 
matrix of rank $2$.
This stress has
$\omega =1$ for the edges between dark blue and bright red, and 
between bright blue and dark red, and $\omega=0$ for the other edges.

\section{Stratification of Stresses}

We want to look at all equilibrium 
stresses for all $d$-dimensional frameworks of $G$, and see which ones 
correspond to those in $\LL$. 
We will do this by slicing the configuration space into 
subsets that are easy to analyze on their own. 
We will find that, when $G$ is generically globally rigid in $\RR^d$,
the equilibrium stresses arising
from infinitesimally flexible frameworks 
can account for only a low dimensional subset of 
$\LL$. This will then lead immediately to a proof
of our Theorem~\ref{thm:main}.

\subsection{The Subsets}

We now  chop up the configuration space, $\RR^{nd}$
into nicely behaved
subsets.
It will be helpful
for each of these subsets, $S$, to be invariant with respect to invertible
affine transforms.
That means that if 
$\p=(\p_1,\p_2..,\p_n) \in S$ 
and $A$ is an invertible affine transform, then 
 $A(\p):=(A(\p_1),A(\p_2)..,A(\p_n)) \in S$.

\begin{lemma}
\label{lem:dropinvar}
Let $S$ be a semi-algebraic set of configurations
that is invariant with respect  to 
invertible affine transforms. 
Let $r$ be the maximal rank of the rigidity matrices over all
$\p \in S$.
Let $S'$ be the semi-algebraic subset of $S$, where
the rigidity matrix has rank less than $r$. Then $S'$ is invariant with respect
to invertible affine transforms.
\end{lemma}
\begin{proof}
The rank of the rigidity matrix of $(G,\p)$ is invariant with respect to invertible
affine transforms acting on $\p$.
\end{proof}

\begin{lemma}
\label{lem:singinvar}
Let $S$ be a semi-algebraic set of configurations
that is invariant with respect to invertible affine 
transforms. Then its singular 
set is invariant with respect 
to invertible affine transforms.
\end{lemma}
\begin{proof}
Each such affine transform gives us a diffeomorphism on $\RR^n$ and thus
retains smoothness of points in subsets.
\end{proof}

We thank Dylan P. Thurston for the proof of the following lemma.

\begin{lemma}
\label{lem:compinvar}
Let $S$ be a reducible semi-algebraic  set of configurations
that is invariant with respect to 
invertible affine transforms. Then each irreducible component 
$S_i$
of $S$ is invariant with respect to invertible affine transforms.
\end{lemma}
\begin{proof}
Let $\AAA$ be the set of invertible affine transforms;
this is  an irreducible semi-algebraic set.
Let $g_i$ be  the map 
$S_i \times \AAA \to S$, which
sends $(\p;A):=(\p_1,\p_2..,\p_n;A)$ to $A(\p):=(A(\p_1),A(\p_2)..,A(\p_n))$.
Since its domain is irreducible,
the image of the polynomial map, $g_i$,
must be an irreducible semi-algebraic set.
Moreover, since the identity map is in $\AAA$, the image of $g_i$ must
contain $S_i$. 
But since a component, by definition, must be maximal,
$S_i$ cannot be a strict subset of this irreducible image. Thus it agrees
with it. Thus $S_i$ is invariant with respect to invertible 
affine transforms.
\end{proof}

Now we describe how we slice up the configuration space.

\begin{definition}
Let $G$ be a graph.
We will consider a framework $(G,\p)$
as a single point in the configuration space, $\RR^{nd}$.
Let \defn{$\IR$} be the set of infinitesimally rigid frameworks of $G$ in $\RR^d$.

Let \defn{$\IF$} be the infinitesimally flexible frameworks, with a 
$d$-dimensional affine span. 
\end{definition}

\begin{lemma}
\label{lem:IR}
The set IR  is a smooth irreducible semi-algebraic 
set, invariant with respect to 
invertible affine
transforms.
\end{lemma}
\begin{proof}
If $G$ is generically infinitesimal rigid in $\RR^d$, then
$\IR$ is a Zariski open subset of configuration space.
Otherwise it is empty. 
Thus $\IR$  is a smooth semi-algebraic set.
From Lemma~\ref{lem:dropinvar} we see that 
IR is invariant with respect to 
invertible affine
transforms.
\end{proof}

Next we split up $\IF$ into nicely behaved smaller sets.

\paragraph{Splitting by rank:}
We take $\IF$ and we partition it using the rigidity matrix rank of each 
$\p$ in $\IF$. 

In particular we start with $\IF$,
which is a semi-algebraic set. 
Let $r$ be the maximum rank of
all of the rigidity matrices in $\IF$. 
Let $\IF'$ be the set of configurations in $\IF$ with 
rigidity matrices of rank $< r$.
We partition 
$\IF = \IF' \cup W$ where 
$W:= (\IF - \IF')$.
As rank dropping can be expressed as an algebraic equality condition, 
both $\IF'$ and $W$ are semi-algebraic sets.
All of the frameworks in $W$ have rigidity matrices of the fixed rank $r$.
From Lemma~\ref{lem:dropinvar} we see that  $\IF'$ and thus also 
$W$ are invariant with respect to 
invertible affine
transforms.

We can then apply the above splitting recursively on $\IF'$.
When this terminates (by descent on $r$), collecting all of the resulting $W$, 
we have partitioned $\IF$ 
a finite set of semi-algebraic subsets
$\{A_1..A_k\}$ for some $k$.
Each 
$A_i$ is
invariant with respect to 
invertible affine
transforms.

\paragraph{Singularity splitting step:}

Given a
semi-algebraic 
set $A$ of dimension $s$, the set can be partitioned,
semi-algebraically, as 
$A = W \cup A'$ where $A':= \sing(A)$ and $W := A-A'$.

By construction, $W$ must be smooth and of dimension $s$.
From Lemma~\ref{lem:singStrat}
the  dimension of $A'$ must be strictly less than $s$.
From Lemma~\ref{lem:singinvar}, if $A$ is invariant
to affine transforms, then so too is $A'$ and thus also $W$.

We can recursively apply this procedure to $A'$.
By descent on $s$,
this process must terminate, giving us a collection of $W$-sets.

Applying this recursive splitting over all of the $A_i$ from the 
singularity splitting step, we collect all of the $W$-sets
to obtain a set of semi-algebraic sets we call
$\{B_1..B_l\}$ for some $l\in \mathbb{N}$.
Each $B$ in 
the resulting collection  must be smooth 
and affine invariant.

\paragraph{Component splitting step:}

Given a semi-algebraic  set $B$, the set can be written uniquely
as the 
finite union of a set
of semi-algebraic irreducible components.
From Lemma~\ref{lem:compinvar}, if $B$ is invariant
to affine transforms, then so too is each of its components.
If $B$ is smooth, then from Lemma~\ref{lem:compSmooth}, so too are 
each of its components.

We collect all of the components over all of the $B_i$ from the 
singularity splitting step
to obtain a set of semi-algebraic sets we call
$\{\IF_1..\IF_m\}$ for some $m$.

We summarize the conclusion of this discussion as follows
\begin{lemma}
$\IF$ can be written as the union of a finite set of semi-algebraic sets
$\{\IF_1..\IF_m\}$ for some $m$, where each $\IF_i$ is  smooth, irreducible
and affine-invariant. 
Each $\IF_i$ is defined over a finite extension of $\QQ$.
All configurations in one $\IF_i$ share their rigidity matrix rank.
\end{lemma}

\begin{definition}
Let \defn{$C_i$} be the codimension of $\IF_i$ within the 
$nd$-dimensional set of configurations.

Let \defn{$F'_i$} be the dimension,
for each framework in 
$\IF_i$,
of its space of infinitesimal  
flexes.
Let \defn{$F_i := F'_i -\binom{d+1}{2}$},
which discounts the dimension of the trivial infinitesimal flexes.
\end{definition}

\begin{definition}
Let \defn{$\ST(\IR)$} and \defn{$\ST(\IF_i)$} be the union of the equilibrium 
stress matrices over its framework
set. 
\end{definition}

\begin{lemma}
The set
$\ST(\IR)$ and each $\ST(\IF_i)$ is semi-algebraic.
\end{lemma}
\begin{proof}
This follows immediately using quantifier elimination.
\end{proof}

\begin{remark}
In the above decomposition, the property of constant rank rigidity matrices
will be used throughout our reasoning.

The invariance to invertible affine transforms will be needed in 
Lemmas~\ref{lem:RisD}
and~\ref{lem:CgreaterF}, where we need to carefully count the dimension
of $\ST(\IR)$ and the $\ST(\IF_i)$.

The irreducibility of each $\IF_i$
will be important in Section~\ref{sec:tang}, where we want 
there to be a well defined notion of a generic point, and thus a 
well defined generic dimension of tangential flexes. 

The smoothness of each $\IF_i$ 
will be convenient throughout, but will be especially needed
in Lemma~\ref{lem:ClessF}, where we will want the generic frameworks of 
$\IF_i$ to be dense in $\IF_i$.
(A non smooth real algebraic set, such as the Whitney umbrella can have a locus
of singular points, such as the handle of the umbrella, that have no nearby 
smooth points.)
\end{remark}

The following Lemma is not needed for the proof of our theorem, but
is useful in setting up the proof of
Lemma \ref{lem:CgreaterF} below.

\begin{lemma}
\label{lem:RisD}
Let $G$ be generically globally rigid in $\RR^d$. Then
$\dim(\ST(\IR))= m-\binom{d+1}{2}=D_L$.
\end{lemma}
\begin{proof}
We want to count the total dimension of equilibrium stresses
over all frameworks in $\IR$, 
but we need to be careful not to double count.

From Theorem~\ref{thm:glr} the dimension of IR is $nd$.

The rank of the rigidity matrix of any framework $(G,\p)$ in IR is 
$nd-\binom{d+1}{2}$, and so the dimension of equilibrium stresses for
this single framework is $m-nd+\binom{d+1}{2}$.

Let the  equilibrium stress bundle of IR, a subset of 
$\RR^{nd} \times \RR^m$, consist of pairs 
$(\p,\omega)$ where $(G,\p) \in \IR$ and $\omega$ is an equilibrium
stress of $(G,\p)$. This is a vector bundle over $\IR$.
The projection, $\pi_2$, of this bundle  onto its second factor, gives us
$\ST(\IR)$.

From Lemma~\ref{lem:bundle1}
the bundle is irreducible and has dimension 
$nd+
[m-nd+\binom{d+1}{2}] = m+\binom{d+1}{2}$.

Let us now look at one
$(\p,\Omega)$, some generic point of 
the bundle.
From Lemma~\ref{lem:image-generic},
the configuration $\p$ must be a generic configuration.
Since $G$ is generically globally rigid, then from
Theorem~\ref{thm:necc}, the generic framework 
$(G,\p)$ must have an equilibrium
stress matrix of rank $n-d-1$. Thus, from genericity,
$\Omega$ must achieve this rank.

Now we look at $\pi_2$, in the  neighborhood 
around
$(\p,\Omega)$. 
Since $\Omega$ has rank $n-d-1$,
the fiber of $\pi_2$
consists of affine transforms of a $\p$
and must have dimension $d(d+1)$. So from Theorem~\ref{thm:fiber},
the dimension of the image must be 
$m+\binom{d+1}{2} - d(d+1)  = m-\binom{d+1}{2}$.

From Theorem~\ref{thm:hen}, $G$ must be $(d+1)$-connected and so
from Corollary~\ref{cor:dl} this dimension agrees with $D_L$.
\end{proof}
\begin{remark}
When $G$ is not
generically globally rigid in $\RR^d$, then all of the 
stresses in $\ST(\IR)$ must be of rank less than $n-d-1$, and thus the
fibers under $\pi_2$ are larger 
and thus
$\dim(\ST(\IR))
< m-\binom{d+1}{2}$.
\end{remark}

\subsection{The $\IF_i$ with few infinitesimal flexes can
only account for low dimensional subsets of $\LL$}

\begin{lemma}
\label{lem:CgreaterF}
Let $G$ be any graph.
Then for any $i$ such that
$C_i > F_i$,
we have $\dim(\ST(\IF_i)) < 
 m-\binom{d+1}{2}$.
Thus, when $G$ is $(d+1)$-connected, then 
$\dim(\ST(\IF_i)) < D_L$.
\end{lemma}
\begin{proof}
We proceed exactly as in the proof of Lemma~\ref{lem:RisD}.

By assumption 
the dimension of $\IF_i$  is $nd-C_i$
and 
the dimension of equilibrium stresses for
any single framework in $\IF_i$ is $m-nd+\binom{d+1}{2}+F_i$.

Let the  equilibrium stress bundle of $\IF_i$, a subset of 
$\RR^{nd} \times \RR^m$, consist of pairs 
$(\p,\omega)$ where $(G,\p) \in \IF_i$ and $\omega$ is an equilibrium
stress of $(G,\p)$. 
The projection, $\pi_2$, of this bundle  onto its second factor, gives us
$\ST(\IF_i)$.

From Lemma~\ref{lem:bundle2}, 
this bundle is irreducible and has dimension 
$[nd-C_i]+
[m-nd+\binom{d+1}{2}+F_i]$ 
which by assumption is strictly less than
$m+\binom{d+1}{2}$. 

The fiber of $\pi_2$ around some generic
$(\p,\Omega)$ includes at least the invertible
affine images of $\p$. Since, by assumption, $\p$ has a full dimensional
affine span, this fiber has  
dimension at least $d(d+1)$.
Thus the dimension of the image of $\pi_2$ is strictly less than 
$m-\binom{d+1}{2}$. 

From Corollary~\ref{cor:dl}, when $G$ is $(d+1)$-connected,
$D_L = m-\binom{d+1}{2}$. 
\end{proof}

\subsection{Any $\IF_i$ with many tangential infinitesimal flexes
cannot account for any stresses in $\LL$}
\label{sec:tang}

\begin{definition}
Let \defn{$T'_i$} be the dimension, at any generic point $\p$ 
in $\IF_i$,
of the space of infinitesimal
flexes in $\RR^d$ of $(G,\p)$
(each thought of a single  vector in $\RR^{nd}$)
that are tangential to the manifold $\IF_i$ at $\p$.
Define \defn{$T_i := T'_i - \binom{d+1}{2}$}, to discount the 
infinitesimal flexes arising from the group $SE(d)$ (these must
all be 
tangential, since $\IF_i$ is invariant with respect to invertible
affine transforms).
Non-generically within $\IF_i$, 
the dimension of tangential infinitesimal flexes
can rise. 

Let \defn{$X_i := F_i-T_i$}.
This quantity represents the dimension, at any generic point $\p$ in $\IF_i$, 
of a linear space of (necessarily non-trivial) infinitesimal flexes that 
is linearly independent from  the tangent space of $\IF_i$ at $\p$.
\end{definition}

\begin{remark}
As mentioned in the previous definition, 
non-generically within $\IF_i$, the dimension of tangential
infinitesimal flexes can rise. One might be tempted to simply refine
our stratification based on this property, cutting out such loci 
into their  own $\IF_i$.  The problem
with this approach is  the 
resulting subdivided $\IF_i$ might not  be invariant to invertible affine
transformations. 

To see the difficulty, suppose $\p$ and $\q$ in some (unsubdivided) $\IF_i$ 
are related by
a $d$-dimensional invertible affine transform with  $\M$ as its linear 
factor.
Then the tangent of $\IF_i$ at $\p$ will map under $\M$ to the tangent 
of $\IF_i$ at $\q$. On the other hand, some flex $\p'$ of $(G,\p)$ will
map to a flex of $(G,\q)$ through the dual map $\M^{-t}$.
\end{remark}

\begin{lemma}
\label{lem:mech}
Let $G$ be  any graph.
For any $\IF_i$ such that  $T_i \geq 1$, every 
generic framework $(G,\q)$ in $\IF_i$ must be locally flexible in $\RR^d$.
\end{lemma}
\begin{proof}
Consider the smooth map 
from $\IF_i$ to $\RR^m$ that measures
squared edge lengths.
The kernel of the linearization of the map at any configuration
$\p \in \IF_i$
consists of the 
infinitesimal flexes for $(G,\p)$ that are tangential to $\IF_i$.
A regular point of this map  is a configuration where the dimension
of the kernel of the linearization of the map
is at its minimum (and generic) value, 
$T_i +\binom{d+1}{2}$. Every generic point in $\IF_i$ is a regular
point of this map.

At a regular point, $\q$, using the 
constant rank theorem, we see that locally, there is a 
$T_i +\binom{d+1}{2}$-dimensional
submanifold of $\IF_i$ that maintains the edge lengths of $\q$.
This fiber gives us our desired non-trivial, finite flex.
\end{proof}

\begin{lemma}
\label{lem:nearby}
If $(G,\p)$ is a framework in $\IF_i$ (resp. IR)
and is super stable, then so too is
any other nearby-enough framework $(G,\q)$ in $\IF_i$ (resp. IR).
\end{lemma}
\begin{proof}
By assumption, $(G,\p)$ has a PSD equilibrium stress matrix
$\Omega$ of rank $n-d-1$.
From Lemma~\ref{lem:nearStress},
any nearby $(G,\q)$ in $\IF_i$ (resp. IR)
must have some equilibrium stress matrix
close to $\Omega$. 
Since eigenvalues vary continuously with the matrix, this
nearby equilibrium stress matrix must retain its $n-d-1$
positive eigenvalues. As an equilibrium stress matrix, 
it cannot gain any more non-zero eigenvalues, and is thus 
PSD.

By assumption, $(G,\p)$ does not have its edges on a conic
at infinity.
Since frameworks with their edges on a conic at infinity are a proper
algebraic subset of configuration space, if $(G,\p)$ does not 
have its edges on a conic at infinity, then neither do 
 nearby frameworks in $\RR^{nd}$.
\end{proof}

\begin{lemma}
\label{lem:ClessF}
Let $G$ be  any graph.
For any $\IF_i$ such that 
$T_i \geq 1$, it must be that
$\ST(\IF_i)$ is disjoint from $\LL$.
\end{lemma}
\begin{proof}
If any 
framework $(G,\p)$
in $\IF_i$ were super stable, then 
from Lemma~\ref{lem:nearby},
so too would
be any nearby enough framework in $\IF_i$.
Since $\IF_i$ is smooth, 
the generic  frameworks of $\IF_i$
are (Euclidean) dense in $\IF_i$ (Lemma~\ref{lem:genDense}).
Thus there would  be a nearby $\q$ which is generic in 
$\IF_i$ and such that $(G,\q)$ is super stable.
But from Lemma~\ref{lem:mech}, 
since $T_i \geq 1$, 
there can be no such $(G,\q)$.
So 
$(G,\p)$
cannot be super stable.

On the other hand, 
from Theorem~\ref{thm:alf}, all 
frameworks (with a $d$-dimensional affine span)
arising as the kernels of the stresses from  $\LL$,
must
be super stable. 
\end{proof}

\subsection{When $G$ is generically globally rigid, 
there can be no $\IF_i$ with many transverse infinitesimal flexes}

\begin{lemma}
\label{lem:ClessX}
Let  $G$ be generically globally rigid in $\RR^d$.
Then for all $i$, 
we have
$C_i > X_i$.
\end{lemma}
\begin{proof}
Otherwise we could apply 
Connelly's global flexibility argument from~\cite{Connelly-K55} and obtain a 
contradiction with the assumed generic global rigidity.
For completeness, we will spell out this argument in detail.

We first record the following principle~\cite[Theorem 6.1]{Connelly-K55}.
\begin{lemma}
\label{lem:pushpull}
Suppose that $\p'$ is an infinitesimal flex for a framework 
$(G,\p)$ in $\RR^d$, where the points of $\p$ do not all lie in a hyperplane.
Then $(G,\p+\p')$ has the same edge lengths as $(G,\p-\p')$.
Moreover, $\p+\p'$ is 
congruent to $\p-\p'$ iff $\p'$ is a trivial infinitesimal flex.
\end{lemma}

From Lemma~\ref{lem:flexmap},
we can define a 
rational map ($f$ stands for flex),
$f(\p,\x): \IF_i \times \RR^{F_i} \to \RR^{nd}$
that (over its domain/where there is no division by zero)
maps to a non-trivial infinitesimal flex $\p' \in \RR^{nd}$ of $\p$,
and for a fixed
$\p$, the map is a linear injective map over $\x$.
Let $\p^0$ denote some configuration that is generic in $\IF_i$,
around which $f$ is well defined.

Given $f$, define the rational map ($o$ stands for offset),
$o(\p,\x) := \p + f(\p,\x)$,
that offsets $\p$ by an infinitesimal flex.
Now we look at the image of the linearization, $o^*$ at
$(\p^0,\0)$.
By varying just the $\p$-variables, we see that this
image contains the tangent space of $\IF_i$ at $\p^0$ (of dimension
$nd-C_i$).  By varying just the $\x$-variables, we see that this image
contains a space of non-trivial flexes of dimension $F_i$. 
Since 
$\p^0$ is generic in $\IF_i$,
this space contains a linear space of dimension $X_i$
that is linearly independent from the tangent space of $\IF_i$ at
$\p^0$.  Thus the image of the linearization is of dimension at least
$nd - C_i + X_i$ (and no greater than $nd$). 
When $C_i \leq X_i$, this rank is $nd$
and so we have a local submersion. As a result, the image of $o$ has
dimension $nd$. Thus we have shown that a full dimensional subset of
configurations can be reached by starting with an infinitesimally
flexible framework in $\IF_i$ and adding to it some non-trivial 
infinitesimal flex.

From Lemma~\ref{lem:pushpull},
when $\p$ has a
full $d$-dimensional affine span (as  all frameworks in $\IF$ do by
assumption) and $\p'$ is a non-trivial infinitesimal flex, then
$\p+\p'$ must be not globally rigid in $\RR^d$.
But our construction has found a
full dimensional set of such configurations, which contradicts the
assumed generic global rigidity.  
\end{proof}

\subsection{Putting the cases together}

Having dealt with all the possibilities, we arrive at our main proposition:

\begin{proposition}
\label{prop:main}
If $G$ is generically globally rigid in $\RR^d$, 
then a full dimensional subset of
$\LL$ is contained in 
$\ST(\IR)$.
\end{proposition}
\begin{proof}
The sets,
$\IR$ and $\IF_i$ cover  all configurations with a full affine span,
and thus the union $\ST(\IR)$ and the $\ST(\IF_i)$ must contain all of 
equilibrium stresses of
all configurations with full affine spans.
Meanwhile, from Theorem~\ref{thm:alf}, all of the stresses in
$\LL$ arise as equilibrium stress matrices of
configurations in general affine position and thus
with full affine spans.
Thus  $\LL$ must be
contained in the union $\ST(\IR)$ and $\ST(\IF_i)$. 

First we show that $\LL$ must be disjoint from 
the $\ST(\IF_i)$ where $C_i \leq F_i = X_i +T_i$. 
When $C_i \leq F_i$ then either (a) $C_i \leq X_i$ or (b) $T_i \geq 1$.
The  case (a) cannot occur at all due to Lemma~\ref{lem:ClessX}. In 
case (b), $\LL$ must be disjoint from $\ST(\IF_i)$ due to Lemma~\ref{lem:ClessF}.

Next we look at the $\ST(\IF_i)$ where $C_i > F_i$.
From Theorem~\ref{thm:hen}, $G$ is  $(d+1)$-connected.
But then from Lemma~\ref{lem:CgreaterF}, 
these $\ST(\IF_i)$ are of lower
dimension than 
$\LL$. Thus only a low dimensional subset of $\LL$ can be contained in these
$\ST(\IF_i)$.

Thus a full dimensional subset of $\LL$ must not be contained in the union
of the $\ST(\IF_i)$ and thus must be 
be contained in $\ST(\IR)$.
\end{proof}

\begin{remark}
When $G$ is  $(d+1)$-connected but not generically globally rigid in $\RR^d$,
then from the above discussion we see that almost all of the stresses
in $\LL$ must come from $\IF_i$ where
$C_i \leq X_i$.
\end{remark}

And now we can prove our main theorem.

\noindent {\bf Proof of Theorem~\ref{thm:main}.}
From Proposition~\ref{prop:main}, there must be an $\Omega$ that is both 
in $\LL$ and in $\ST(\IR)$. Thus there must be a framework $(G,\p)$
which is infinitesimally rigid and has $\Omega$ as one of its
equilibrium stress matrix.
From Theorem~\ref{thm:alf}, $\Omega$ must be 
PSD of rank $n-d-1$, and so
$(G,\p)$ must be super stable.

For the second part, we use Lemma~\ref{lem:nearby} to conclude that
any nearby framework in IR must be super stable. As IR is full dimensional
and open, any nearby framework in configuration space must
be infinitesimally rigid and super stable. Such a neighborhood 
contains a generic configuration.
\qed

\section{The Stress Variety}
\label{sec:stress}
The main theorem in this paper relates to a deeper question about 
the algebraic set of stress matrices.
As above, 
let $G$ be a  graph with $n$ vertices and $m$ edges,
and $d$ a fixed dimension. 

\begin{definition}
Let \defn{$\ST$} be the real algebraic set of $n$-by-$n$
\defn{$d$-dimensional stress matrices} for $G$.
Specifically, this is the set of real symmetric matrices that have $0$ entries
corresponding to non-edges of $G$,  with the all-ones vector in its
kernel, and with  rank $n-d-1$ or less.
\end{definition}
The set $\ST$ is the union of $\ST(\IR)$ and all of the $\ST(\IF_i)$
described above. In particular since the kernel of an $\Omega \in \ST$ is
of dimension at least $d+1$, 
we can always pick a framework $\p$ with a $d$-dimensional
affine span with spatial coordinates in this kernel. Clearly $\Omega$ is an 
equilibrium stress matrix for $\p$.

\begin{question}
Suppose that $G$ is generically globally rigid in $\RR^d$. Is its associated
$d$-dimensional stress variety, $\ST$,  irreducible?
\end{question}

There are some results in the literature about the irreducibility of
certain linear sections of determinantal varieties~\cite{merle,eisen}, but
these do not appear to be strong enough to answer the present question.

The irreducibility of $\ST$ would be useful, since any 
strict algebraic subset $W$ of an irreducible (semi-)algebraic set 
must be of strictly lower dimension!
In particular, an affirmative answer to this question would then lead to an 
alternative direct proof
of Theorem~\ref{thm:main}, which we now sketch:

As described in Section~\ref{sec:kerstress},
we can select a rational map that maps from a matrix
$\Omega \in \ST$ 
to a 
framework in its kernel with a $d$-dimensional affine span. 
In the image of 
this map, $d+1$ chosen vertices will always lie in some pinned positions.
The map will be undefined 
over some subvariety $V$ of $\ST$.
(The subvariety $V$ consists of 
all of the $\Omega$
of rank strictly less than $n-d-1$ and any
$\Omega$ of rank $n-d-1$ which is an
equilibrium 
stress matrix of a $d$-dimensional framework
where the chosen $d+1$ vertices lie in a single hyperplane.)  

The preimage of the algebraic
set, $\IF$, must lie in some algebraic subset $W$ of $\ST$. By
Theorem~\ref{thm:necc}, this subset of $\ST$ is strict.  

Suppose that there is an $\Omega \in \ST(\IF)$ that has rank $n-d-1$ and
is the equilibrium stress
of an infinitesimally flexible framework $(G,\p)$ with
the chosen $d+1$ vertices in general affine position. Then our rational map
must map $\Omega$ to a configuration $\q$ which is an affine transform
of $\p$. The framework $(G,\q)$ must be in $\IF$, and thus 
$\Omega \in W$.

Thus $\ST(\IF)$ must
lie in the union of $V$ and $W$.  
The rest of the matrices,
$\ST-(V\cup W)$, must be in $\ST(\IR)$.
If $\ST$ is irreducible, the subset $V \cup W$
must be of strictly lower dimension than $\ST$ itself.  

Meanwhile, Lemma~\ref{lem:RisD} tells us that the dimension of $\ST(\IR)$
is $m-\binom{d+1}{2}$.  Thus the dimension of $\ST(\IF)$ is strictly less
than $m-\binom{d+1}{2}$ and thus less than $D_L$. And we are done.\qed
\vspace{2 mm}

Our question is also related to one posed by 
Lov\'asz et al~\cite{Lovasz-Schrijver}.
Recalling the definition of an OR 
from Definition~\ref{def:GOR}
they ask: under what conditions is the set of ORs irreducible?
Some progress on this question is reported in~\cite{herz}.

\section{Graph realization SDP}
Our results  relate to the problem of 
finding a framework $(G,\p)$ with a specific 
set of desired edge lengths.

\begin{definition}
Let $(G,\p)$ be a $d$-dimensional framework, and let 
\[
\ell = (\ell_{ij})_{\{i,j\}\in E(G)} := (|\p_i - \p_j|^2)_{\{i,j\}\in E(G)}
\]
be the vector of squared edge length measurements. 
The \defn{graph realization problem} is to find $(G,\p)$ given 
$G$, $\ell$ and $d$. This problem is is 
NP-hard \cite{Saxe}.
An instance is well-posed if and only if $(G,\p)$ is 
globally rigid.
\end{definition}
Due to its wide applicability, graph realization, and related
``distance geometry problems'', 
have received a lot of attention.  
See the survey \cite{dist-survey} 
for an overview.
Given the problem's hardness, 
practical algorithms will be
heuristic
\footnote{When $G$ is $K_n$, one rigorous 
notion of an approximate solution is a low-distortion embedding (see, e.g., \cite[Chapter 15]{M02-discrete}
or \cite{L02-embedding}).  This is a bit 
different in flavor from distance geometry where the dimension constraint and being exact on the 
given distances are most important.} in nature, or 
involve restricting $G$
to some class that is smaller than being 
generically globally rigid.  
An important practical approach   
is based on semidefinite
programming (see 
\cite{VB96} for a general overview of SDP).
\begin{definition}
Let $S_n$ be the cone of symmetric $n\times n$ real matrices.
Let $S^+_n$ be the cone of symmetric $n\times n$ positive semidefinite 
matrices, and define an inner product on $n\times n$ matrices
by $\iprod{X}{Y} := \operatorname{Tr}{X^t Y}$.  A \defn{semidefinite
program} (SDP) 
is an optimization problem of the form
\[
	\inf_X \{\iprod{X}{\beta} : X\in S^+_n\cap (L + b)\}
\]
where $b$ and $\beta$ are in $S_n$ and $L$ is a linear suspace of $S_n$.
Semidefinite programming is a convex problem that can be approximated in polynomial time.
\end{definition}
A semidefinite program for graph realization 
has been studied for some time  (see \cite{johnson1995connections,Alf-wolk-sdp,LLR95}; \cite[Section 4]{dist-survey} and the references
there).
\begin{definition}
Given a framework $(G,\p)$ and its edge measurement vector $\ell$, 
define $A$
to be the space of $n\times n$ matrices $X$ such that 
$X_{ii} + X_{jj} - 2X_{ij} = \ell_{ij}$ for all edges $\{i,j\}\in E(G)$.
(Notice that $A$ is affine.)

The \defn{graph realization semidefinite program} is 
\[
	\inf_X \{\iprod{X}{0} : X\in S^+_n\cap A\}	
\]
By treating $X$ as the Gram matrix of $\p$ we can recover $\p$ from $X$.

We say that the graph realization SDP \defn{succeeds on $(G,\p)$} if the 
only feasible points of the SDP for the associated problem correspond to 
configurations congruent to $\p$;
otherwise we say that it \defn{fails}.  (Remember that we will only 
get a numerical approximation to $\p$ from an SDP solver.)
\end{definition}
\begin{remark}
The presentation above follows that in 
\cite{Gortler-thurston2}. 
\end{remark}
The graph realization SDP is a convex relaxation of
the rank constraint \cite{schoen35} on a Gram matrix for a $d$-dimensional
point set.  As the description suggests, it is not difficult to 
implement, and, when it succeeds, will ``guess'' the 
correct dimension $d$.  When it fails, solvers based 
interior point methods \cite{PW00} will return a
higher dimensional solution.  Thus, it is interesting to 
know, from $G$ only, whether the SDP can succeed on any positive 
measure set of $\p$.

A connection to universal rigidity was made by 
Zhu, So, and Ye \cite{Ye-So} 
(building on work of So and Ye \cite{Ye-So-sensor}).  
\begin{theorem}
\label{thm:gt-sdp}
Let $(G,p)$ be a generic $d$-dimensional framework with edge 
measurement vector $\ell$.  The graph realization 
SDP succeeds on the graph realization instance given by $G$, 
$\ell$ and $d$ if and only if $(G,\p)$ is universally rigid.
\end{theorem}
Combining Theorem \ref{thm:gt-sdp} with our Theorem \ref{thm:main},
we obtain.
\begin{corollary}\label{cor:sdp}
Let $G$ be a graph and fix a dimension $d$.  Then there is a Euclidean 
open set of frameworks $(G,\p)$ for which the graph realization 
semidefinite program succeeds if and only if $G$ is generically 
globally rigid.
\end{corollary}
Since we know that universal rigidity is not a generic property,
this result is, in a sense, a tight description of which 
combinatorial types of framework the semidefinite programming 
algorithm succeeds on.  (For example, if we draw $\p$ from a 
continuous density, Corollary \ref{cor:sdp} implies that 
the semidefinite program has a positive probability of 
success  if and only if $G$ is generically globally rigid.)

Characterizing the graphs for which every generic $(G,\p)$ 
is universally rigid, and thus the semidefinite 
relaxation is tight with probability one, is 
an open problem.

\newpage
\appendix{}

\section{Algebraic geometry background}

Throughout this paper, we will be using some basic facts about real 
algebraic and semi algebraic sets.
Here we summarize some preliminaries 
from real
algebraic geometry, somewhat specialized to our particular case.  For
a general reference, see, for instance, the
books~\cite{Roy,RoyAlg}. Much of this is adapted
from~\cite{Gortler-Thurston}. We will spend a bit of time
dealing explicitly with some issues of defining fields 
and genericity
as these issues
are not fully covered in any single elementary text.

\begin{definition}
  \label{def:generic-gen}
   Let $\kk$ be a subfield of $\RR$.
  An (embedded) affine, real 
\defn{algebraic set} or \defn{variety}~$V$ 
  defined over $\kk$
is a subset of $\RR^n$ 
that can be 
  defined by a finite set of algebraic equations with coefficients in~$\kk$.

A \defn{Zariski open set} is a subset of $\RR^n$ defined by removing an
algebraic subset.

  A real algebraic set has a 
real \defn{dimension}
  $\dim(V)$, which we will define as the largest $t$ for which there
  is an open subset of~$V$, in the Euclidean topology, that is 
is a smooth $t$ dimensional smooth sub-manifold of $\RR^n$.

Any nested sequence of strict algebraic subsets must terminate
in a finite number of steps. 
(This is called the Noetherian property).
This means that if we continue to take strict algebraic subsets,
we must eventually be left with the empty set.

  An algebraic set is \defn{irreducible} 
if it is not the union of two proper
  algebraic subsets
  defined over $\RR$.

Any reducible algebraic set $V$ can be 
uniquely described as the union of 
a finite number of maximal irreducible algebraic subsets called
the \defn{components} of $V$. 

Any algebraic subset of an irreducible algebraic set must be of 
strictly lower dimension.

\end{definition}

\begin{lemma}
\label{lem:geomIrr}
If a real algebraic set $V$ is defined over $\kk$, a subfield 
of $\RR$, 
then
any of its components can be defined over a finite extension of $\kk$,
also a subfield of $\RR$.
\end{lemma}
\begin{proof}
Let us define $V^*$, the
\defn{complex Zariski closure} 
of $V$, to be 
the smallest algebraic subset of $\CC^n$, defined
by polynomials with complex coefficients, that contains $V$.

If $V$ is defined over $\kk$, so 
too is $V^*$~\cite[Lemma 6]{whitney}.
Each 
component of $V^*$
is defined over a 
subfield of the reals~\cite[Lemma 7]{whitney}. 
Each component of a real algebraic set $V$
is simply the real locus of a corresponding 
component of $V^*$~\cite[Lemma 7]{whitney}. 
Thus our Lemma reduces to 
understanding the defining field
of the components of $V^*$.

Meanwhile it is standard 
fact from 
scheme theory, that 
given a complex variety $V^*$ defined over  $\kk$, some subfield of 
$\CC$, its irreducible components are
themselves defined over some finite extension of $\kk$.
In particular, from~\cite[Tag 038I]{stacks-project}, it suffices to 
just look at the components that are irreducible 
when working over an algebraic 
closure of $\kk$.
Then~\cite[Tag 04KZ]{stacks-project}, tells us that each of these 
components is defined over some finite extension of $\kk$.

\bs{
Meanwhile, a complex variety defined over $\kk$
splits into components iff
it splits over some finite extension of $\kk$
(see~\cite[Prop 5.50]{gortz}).
Thus every component of 
a complex algebraic set defined over $\kk$, a subfield of $\CC$, 
can always be defined over some 
finite extension of $\kk$.}
\end{proof}

\begin{definition}
  A \defn{semi-algebraic set}~$S$ defined over $\kk$ 
  is a subset of $\RR^n$ that can be defined by 
 a finite set of   algebraic
  equalities and inequalities with coefficients in $\kk$, as well as
a finite number of Boolean operations.
  A semi-algebraic set has a well defined (maximal) dimension~$t$  
which we will define as the largest $t$ for which there
  is an open subset of~$S$, in the Euclidean topology, that is 
a smooth $t$ dimensional sub-manifold of $\RR^n$.

Any algebraic set is also a semi-algebraic set.

A semi-algebraic set is comprised of a finite number of 
connected components~\cite[Theorem 2.4.4]{Roy}.

  The real  \defn{Zariski closure} of $S$ is the smallest
real   algebraic set defined over $\RR$
containing $S$. 

\bs{
(We may 
need to enlarge the  field to cut out  the smallest algebraic set
containing~$S$
but a finite extension will always suffice
In particular, the real Zariski closure of 
$S$ might be only a component
of the smallest algebraic set defined over $\kk$ that contains $S$.)
}

We call $S$ \defn{irreducible} if its real Zariski closure is
  irreducible.

A semi-algebraic set $S$ has the same real dimension as its
real Zariski closure 
(see~\cite[Prop 2.8.2]{Roy} or~\cite[Lemma 2]{stasica}).  
Thus if two irreducible semi-algebraic
sets of the same dimension have an intersection of that same dimension,
then their union must be irreducible.

Any reducible semi-algebraic set  $S$ can be 
uniquely described as the union of 
a finite number of maximal irreducible semi-algebraic sets  called
the \defn{components} of $S$. 
Each component of $S$ is the intersection of $S$ with a component of 
the real Zariski closure of $S$.
Thus, if $S$ is defined over $\kk$, than
any of its components can be defined over a finite extension of $\kk$.

The image of a  real semi-algebraic  (or algebraic set)
set under a polynomial or rational  map, all defined over $\kk$
is semi-algebraic and defined over $\kk$~\cite[Theorem 2.76]{RoyAlg}.
As a corollary to this, quantifiers can always be eliminated from any first-order
formula over the reals, involving polynomial equalities and
inequalities,
rendering its feasible set
semi-algebraic.

The image of a  real semi-algebraic
set under an injective polynomial or rational  map
has the same dimension as its domain (see~\cite[Prop 2.8.8]{Roy}).
The image of as irreducible real  semi-algebraic
set under a polynomial or rational  map
is  irreducible (see the proof of~\cite[Prop 2.8.6]{Roy}).

We call a point on $S$
\defn{smooth}  
if it has a neighborhood in $S$ 
that is a smooth sub-manifold of $\RR^n$ of dimension $\dim(S)$.
(In the semi-algebraic setting, any such smooth sub-manifold will also be a 
real analytic sub-manifold of $\RR^n$ (see~\cite[Prop 8.1.8]{Roy}).

If all points of $S$ are smooth, then $S$ is called smooth.

Any semi-algebraic set $S$ of dimension $t$,
defined over $\kk$ can be stratified into   the 
finite disjoint union of smooth semi-algebraic sets of various dimensions
defined over $\kk$. In the stratification, the Euclidean closure in $S$ of
one stratum consists of itself and some
lower
dimensional strata.
(See~\cite[Theorem 5.38]{RoyAlg} for 
a detailed description.) 

The smooth and non-smooth loci of 
points of $S$ form  semi-algebraic sets, also defined
over $\kk$ 
(the proof of semi-algebraicity in ~\cite{stasica}, 
shows how these loci can be defined using quantifier elimination,
which establishes that $\kk$ is a defining field).

\end{definition}

\begin{lemma}
\label{lem:singStrat}
Let $S$ be a semi-algebraic set of dimension $t$. Then its singular
locus has dimension $<t$.
\end{lemma}
\begin{proof}
In any smooth stratification of $S$, points that are not smooth 
cannot lie in a top-dimensional stratum. 
(See also~\cite[Prop 5.53]{RoyAlg}).
\end{proof}

\begin{lemma}
\label{lem:smoothIrr}
Let $S$ be a smooth and connected semi-algebraic set.
Then its real Zariski closure is irreducible.
\end{lemma}
(See~\cite[Prop 8.4.1]{Roy}.

\begin{lemma}
\label{lem:compSmooth}, 
Let $S$ be a smooth  semi-algebraic set.
Then each of its irreducible components is smooth.
\end{lemma}
\begin{proof}
From Lemma~\ref{lem:smoothIrr}, if $B$ is smooth and connected, then it
must, itself be irreducible. 

In general,  $S$ 
might consist of
some finite number of disjoint connected components. In this case, each
of the irreducible components of $S$ consists exactly of those connected
components that have a common real Zariski closure. 

Thus an irreducible component is the union of disjoint 
smooth semi-algebraic sets
and is thus smooth.
\end{proof}

\bs{
\begin{lemma}
\label{lem:smoothIrr2}
Let $S$ be an irreducible 
semi-algebraic set of dimension $t$. Then its smooth locus is 
irreducible.
\end{lemma}
\begin{proof}
Let $W$ denote the smooth locus of $S$.
It has dimension $t$. 
Thus the real Zariski closure of $W$ 
must also be of dimension $t$ and be contained inside of
the real Zariski closure of $S$.
Since $S$ is irreducible, these two real Zariski closures must agree, making
$W$ irreducible. 
\end{proof}
}

\begin{lemma}
\label{lem:zclose}
Let $S$ be a semi-algebraic set of dimension $t$
defined over $\kk$. Then the real
Zariski closure of $S$ is defined (as a variety) over a finite extension
of $\kk$.
\end{lemma}
\begin{proof}
Due to the finite stratification, we can assume that $S$ is smooth 
and connected, 
of some dimension $s$. (Then we can just take the finite 
union over these strata, as the closure of their union is the union
of their closures.)

From Lemma~\ref{lem:smoothIrr}, the real Zariski closure $V$, of $S$
is irreducible and of dimension $s$. 
Meanwhile $S$ must be contained in some algebraic set $W$, that has
dimension $s$ and is defined over $\kk$
(see~\cite[Lemma 2]{stasica}). As the real Zariski closure of $S$
must be contained
in any algebraic set containing $S$,
we must have $V \subset W$.
But since $V$ and $W$ have the same dimension, $V$ must be a
(maximal) 
component 
of
$W$ (any algebraic set that is a strict subset of an irreducible
component of $W$ must be of lower dimension).
 From Lemma~\ref{lem:geomIrr}
components of algebraic sets can always
be defined using a finite
extensions, thus we are done.
\end{proof}

\begin{definition}
Let  $\kk$ 
be a countable subfield of $\RR$.
  A point in
  an irreducible  (semi-)algebraic set~$V$ defined over~$\kk$
is
  \defn{generic} if its
  coordinates do not satisfy
  any algebraic equation with coefficients in~$\kk$
  besides those that are satisfied by every
  point on~$V$ (such equations are called trivial).

A point that satisfies some non-trivial algebraic  equation with 
coefficients in $\kk'$, some finite extension of $\kk$, 
will always also satisfy 
some non-trivial algebraic equation with coefficients in~$\kk$
(see e.g.~\cite[Lemma 23]{cgr}).
Thus a point will remain generic when a  finite field extension
is applied to the defining field, which might occur, say, when passing from
a semi-algebraic set to its real Zariski closure or when splitting an algebraic set into
its components.

Almost every point in an irreducible (semi-)algebraic set $V$ defined over
$\kk$
is generic. 
 \end{definition}

\begin{lemma}
\label{lem:genSmooth}
Every generic point of an irreducible (semi-)algebraic set,
defined over a countable field $\kk$,
 is smooth.
\end{lemma}
\begin{proof}
From Lemma~\ref{lem:singStrat}, any non-smooth point lies in a lower
dimensional semi-algebraic set defined over $\kk$,
which 
remains so after a Zariski
closure. Thus these points must satisfy some extra equation,
defined over a finite extension of $\kk$, 
that is 
non-trivial over $S$.
\end{proof}

\begin{lemma}
\label{lem:genDense}
Let $V$ be an irreducible smooth (semi-)algebraic set, defined over 
a countable field $\kk$.
Then its generic points 
 are (Euclidean) dense in $V$.
\end{lemma}
\begin{proof}
Let $\phi$ be any non-zero algebraic function on $V$. Its zero set 
is closed and of dimension lower than that of $V$ and thus is 
stratified as a union 
of finite number of smooth manifolds, each with dimension
less than that of $V$. Since $V$ is a smooth manifold, 
$V_\phi$,  the complement of 
this zero set is open 
and dense (in the subspace topology) 
in $V$.
The generic points are the intersection of $V_\phi$ as $\phi$ ranges
over the countable set of all possible $\phi$ defined over $\kk$.
Since $V$ is a Baire space, such a countable intersection of
open and dense subsets must itself be a dense subset.
\end{proof}

\begin{lemma}
  \label{lem:image-generic}
  Let $V\!$ and $W\!$ be irreducible semi-algebraic sets and $f : V \to W$
  be a surjective polynomial or rational map, all defined over~$\kk$,
  a countable subfield of $\RR$.
Then if $x
  \in V\!$ is generic, $f(x)$ is generic inside~$W\!$.
\end{lemma}
\begin{proof}
Consider any non-zero algebraic function~$\phi$ on~$W$
defined over~$\kk$.  Then 
$\phi(f(\cdot))$ is a function on~$V$ that is not identically
zero.  Thus if $x$ is a generic point in $V$, $\phi(f(x)) \ne
0$.  Since this is true for all~$\phi$, it follows that $f(x)$ is
generic.
\end{proof}

\begin{theorem}
\label{thm:fiber}
Let $S$ be an irreducible semi-algebraic set.
Let $\pi$ be a polynomial or rational
map from $S$ into $\RR^n$, for some $n$, all defined over $\kk$,
a countable subfield of $\RR$.
Let $s$ be a generic point in $S$, and $N(S)$ a sufficiently small
Euclidean neighborhood of $s$ in $S$. 
Then $\dim(S) = \dim(\pi(S)) + \dim(\pi^{-1}(\pi(s)) \cap N(s))$.
\end{theorem}
\begin{proof}
From genericity and Lemma~\ref{lem:genSmooth},
$N(s)$ is smooth, and the rank of the linearization, $\pi^*$
is of constant rank. Thus by the constant rank theorem, we have
$\dim(N(s)) = \dim(\pi(N(s))) + \dim(\pi^{-1}(\pi(s)) \cap N(s))$.

Since $N(s)$ is smooth, we have $\dim(N(s)) = \dim(S)$.

Next we argue that $\dim(\pi(N(s)))=\dim(\pi(S))$, as follows.
If $\dim(\pi(N(s)))$ were smaller than $\dim(\pi(S))$, then 
the  semi-algebraic
set, $\pi(N(s))$  could be cut out from  $\pi(S)$
by a  non-trivial algebraic equation (as the real Zariski closure of 
$\pi(N(s))$ would be of lower dimension than that of the real Zariski closure of 
$\pi(S)$).
This then  means that $N(s)$ could be cut
out of $S$ by a non-trivial algebraic equation. But a full dimensional subset
cannot be cut out from an irreducible semi-algebraic set by an 
algebraic equation that doesn't identically vanish.
\end{proof}
\begin{remark}
Indeed 
it also can be shown (say using algebraic 
Sard's theorem on a smooth stratification of $S$) 
that at generic $s$,
$\dim(S) = \dim(\pi(S)) + \dim(\pi^{-1}(\pi(s)))$, which also means that
$\dim(\pi^{-1}(\pi(s))) = \dim(\pi^{-1}(\pi(s)) \cap N(s))$. 
But we will not need this.
\end{remark}

\section{Rational maps to kernels of matrices}
A number of times in this paper, we will have some algebraic set
$S$  of 
$n_1$-by-$n_2$ 
matrices,
and we will want to construct a map takes an  $\M \in S$ to some vector
in the kernel of $\M$. Let $r$ be the maximal rank over the matrices in $S$.
Here we will outline the general procedure and then
work out the specific maps that are used in this paper.

We  start by choosing some matrix $\M^0 \in S$ with rank $r$.
We then select $r$ rows of $\M^0$ that are linearly independent.
We then find an $(n_2-r)$-by-$n_2$ matrix $\H$ (with entries in $\QQ$) 
such that the selected $r$ rows of
$\M^0$ together 
with the added rows from $\H$, form a non-singular matrix.

For any $\M \in S$, let $\M'$ be the square matrix obtained by using the
same chosen
$r$ rows from $\M$, vertically 
appended with the matrix $\H$ chosen above. The matrix
$\M'$ can only
be singular over some strict subvariety of $S$. 
(When $S$ is irreducible,
singularity can only happen for non-generic $\M$.)

Given any vector $\x \in \RR^{n_2-r}$, we define the vector 
$\v(\x) \in \RR^{n_2}$ as a vector with $r$ leading zeros appended to 
$\x$.
We now define a rational map from $S\times\RR^{n_2-r} \rightarrow \RR^m$
which maps $(\M,\x) \mapsto (\M')^{-1}\v(\x)$. 
Clearly, this map can be expressed using rational functions of the
coordinates of $\M$ and $\x$.
The map is not defined wherever
$\M'$ is singular. Wherever the map is defined, its maps to some vector
in the kernel of $\M$. For a fixed $\M$, the map is linear and injective
over $\x$. 

This procedure can be used to construct
a bundle of matrices together with kernel vectors.

\begin{lemma}
\label{lem:bundle0}
Let $S$ be an $d_1$-dimensional irreducible semi-algebraic set of 
$n_1$-by-$n_2$ matrices all of rank 
$r$.
And let  $d_2:=n_2-r$ be the kernel dimension.
Let $B$ be the bundle $(\M,\y)$
with $\M$ a matrix in $S$ 
of rank $r$ and $\y$ in the kernel of $\M$. 
Then $B$ is an irreducible semi-algebraic set of dimension $d_1+d_2$.
\end{lemma}
\begin{proof}
We use
the general construction described above (starting with a chosen
$\M^0$.
we can build an injective
rational map from
$(S \times \RR^{d_2})$ to $B$, of the form $(\M,\x)
\mapsto (\M,\y)$.  
The image of each such rational map
is an 
irreducible semi-algebraic set, 
and as an injective rational map has dimension
$d_1+d_2$.

Such a rational map may be undefined over  
some subvariety $V$ of $S$, where the constructed linear system
becomes singular,
and thus its image 
may miss some subvariety of $B$. 
But we can always pick a different rational map, (that uses, perhaps a different
set of rows, and perhaps a different $\H$ matrix)
by starting with another $\M^0$, this time in $V$,
so that the map is undefined over a different subvariety $W$.
Since $S$ is irreducible, $V$, $W$, and their union, must be of
lower dimension than $S$, and thus the region of $S$ where
both maps are defined is full dimensional. 
For any matrix where the map is defined, the image,
as we vary $\M$,  is the entire
fiber above that matrix in $B$. Thus if two maps  are defined
over a full dimensional region of $S$, then their images 
in $B$  must have a full
dimensional intersection.
Thus
the union of these images must itself be irreducible.

The subset of $S$ that is not defined under either map,
$V \cap W$, is a strict algebraic subset of $V$.
Due to the Noetherian property, a finite number of such rational maps
is then guaranteed to have regions of definition that cover $S$, and thus
images that cover $B$, which thus must be irreducible.
\end{proof}

In the next sections, we will use variations on this construction.
We alter the construction when we need specific properties of the
kernel vectors.

\subsection{Stresses of Frameworks in $\IR$ or $\IF_i$}

In Lemma~\ref{lem:RisD},
we wish to understand the structure of 
the  
equilibrium stress bundle over $\IR$, a subset of 
$\RR^{nd} \times \RR^m$ consisting of pairs 
$(\p,\omega)$ where $(G,\p) \in \IR$ and $\omega$ is an equilibrium
stress vector of $(G,\p)$. 
This is a vector bundle over $\IR$.
We will do this by looking at a rational map that maps from a
framework to each of its equilibrium stresses. 

\bs{
Since we want to understand
the complete bundle, we may need to use a finite number of such maps, so 
that at least one map is defined at every framework in $\IR$.}

From Theorem~\ref{thm:glr} the dimension of IR is $nd$.

The rank of the rigidity matrix of any framework $(G,\p)$ in IR is 
$nd-\binom{d+1}{2}$, and so the dimension of equilibrium stresses for
this single framework is $m-nd+\binom{d+1}{2}$. 

Then Lemma~\ref{lem:bundle0} 
applied to the co-kernel of the rigidity
matrices gives us:
\begin{lemma}
\label{lem:bundle1}
Assume $G$ is generically infinitesimally rigid.
The equilibrium stress bundle of $\IR$
is irreducible and has dimension 
$m+\binom{d+1}{2}$.  
\end{lemma}

Similarly,
in Lemma~\ref{lem:CgreaterF},
we wish to understand the structure of 
the  
equilibrium stress bundle over $\IF_i$.

By assumption 
the dimension of $\IF_i$  is $nd-C_i$
and 
the dimension of equilibrium stresses for
any single framework in $\IF_i$ is $m-nd+\binom{d+1}{2}+F_i$.
Again
we conclude

\begin{lemma}
\label{lem:bundle2}
The equilibrium stress bundle of $\IF_i$
is irreducible and has dimension 
$[nd-C_i]+
[m-nd+\binom{d+1}{2}+F_i]$ 
\end{lemma}

\begin{lemma}
\label{lem:nearStress}
Let $(G,\p^0)$, a framework in 
$\IF_i$ (resp. IR), 
have an equilibrium stress  $\omega^0$.
Then any nearby framework in 
$\IF_i$ (resp. IR) must have  an equilibrium stress  close
to $\omega^0$.
\end{lemma}
\begin{proof}
Using the general construction above, starting with
$\p^0$, 
we can build a
rational map from
$(\IF_i \times \RR^{m-nd+\binom{d+1}{2}+F_i})$ to $\RR^m$, of the form 
$(\p,\x)
\mapsto \omega$, 
where $\omega$ is an equilibrium stress of $\p$, and such
that the map is well defined in a neighborhood of $\p^0$.
For an appropriate $\x^0$, we have 
$(\p^0,\x^0)
\mapsto \omega^0$.
Since this map is continuous, for a nearby 
$\q$ we must have  
$(\q,\x^0)
\mapsto \omega$, where $\omega$ is close to $\omega^0$.
\end{proof}

\subsection{Infinitesimal Flexes of Frameworks in $\IF_i$}

In Lemma~\ref{lem:ClessX}, we will want a rational map that maps
from a framework in $\IF_i$ to a non-trivial
infinitesimal flex of that framework.
At a fixed framework,
by varying the $\x$ parameters, 
we wish the image to be  an $F_i$-dimensional
space of such flexes.

\begin{lemma}
\label{lem:flexmap},
We can define a rational map ($f$ stands for flex), $f(\p,\x): \IF_i
\times \RR^{F_i} \to \RR^{nd}$ that (over its domain/where there is no
division by zero) maps to a non-trivial infinitesimal flex $\p' \in
\RR^{nd}$ of $\p$, and for a fixed $\p$, the map is a linear injective
map over $\x$.
\end{lemma}
\begin{proof} 
Again, we will use our general construction above to define such a map.
But we need to make special care to make sure that the image of the map
does not contain any trivial flexes. This requires a bit of care in 
defining the extra rows to complete  our square matrix, as well as 
how we construct the 
$\v(\x)$ vector.

We proceed as follows:
Pick $\p^0$,
a 
generic configuration of $\IF_i$.
Pick a subset of
$nd-\binom{d+1}{2}-F_i$ edges that are independent in $(G,\p^0)$. 
Instead of completing this matrix with a constant $\H$ matrix, we do
the following:
Add
a set of $F_i$ ``fake edges'' to this subset to create a graph $G'$ so
that $(G',\p^0)$ is infinitesimally rigid.  (This can be done, one by
one, as $\p^0$ has a full $d$-dimensional affine span).  Finally create $\J$, an
$\binom{d+1}{2}$-by-$nd$ matrix of constants
so that the rows of $\J$  form a
linear complement to the rows of the rigidity matrix of $(G',\p^0)$.

For any $\p$ we define its $nd$-by-$nd$ modified,
non-singular, rigidity matrix $\M'(\p)$, as the rigidity matrix of
$(G',\p)$ with the added 
rows of the $\J$ matrix above appended to it.  We
invert this to obtain $(\M'(\p))^{-1}$.

Let $\v(\x)$ be the $nd$-vector with $nd-\binom{d+1}{2}-F_i$ leading
zeros, followed by the $F_i$ coordinates of $\x$, followed by
$\binom{d+1}{2}$ more zeros.  
Then $f(\p,\x) := (\M'(\p))^{-1}\v(\x)$ gives
us our desired map. Any non-zero coordinates in $\x$ will ensure
that some fake edge changes its length at first order, thus making
our obtained flex, non-trivial.
\end{proof}

\subsection{A framework in the Kernel of a Stress Matrix}
\label{sec:kerstress}
In section~\ref{sec:stress} we want a rational map that maps from
a stress matrix $\Omega$ to a framework in it kernel
with a $d$-dimensional affine span.
We can represent such a framework as a vector in $\RR^{nd}$.
We can represent the kernel condition using the 
$nd$-by-$nd$ matrix $\I_d \bigotimes \Omega$. 

Again, we can then apply our general construction above. 
In order to obtain a framework with a full $d$-dimensional affine span,
we set the extra $\H$ rows to represent the pinning of 
$d+1$ specific vertices. We fix the non-zero elements of the right hand 
side, $\v$, to place these pinned vertex in general affine position.
In this setting, we only want one framework per stress, so there are no
free variables $\x$.

The map will be undefined 
over some subvariety $V$ of $\ST$, 
(which includes, for example
all of the $\Omega$
of rank strictly less than $n-d-1$, and all of the equilibrium 
stress matrices of frameworks
where the chosen $d+1$ vertices lie in a single hyperplane).  

\newpage

\def\v{\oldv}


\end{document}